\documentclass[11pt,reqno]{amsart}
\usepackage[margin=1in]{geometry}
\usepackage{amsmath,amssymb,amsthm}
\usepackage{enumitem,comment}
\usepackage{hyperref}
\usepackage{graphicx}
\usepackage{xcolor}
\usepackage{seqsplit}
\usepackage[T1]{fontenc} 
\usepackage[utf8]{inputenc}

\newtheorem{theorem}{Theorem}[section]

\newtheorem{lemma}[theorem]{Lemma}
\newtheorem{corollary}[theorem]{Corollary}

\theoremstyle{remark}
\newtheorem{remark}[theorem]{Remark}
\newtheorem{example}[theorem]{Example}

\DeclareMathOperator{\Res}{Res}

\title[Verlinde sums, graph theory and Binet formulas]{Evaluating $\text{\rm SU}(3)$ Verlinde sums using spectral graph theory}
\author{Jay Jorgenson}\thanks{The first named author acknowledges grant support from PSC-CUNY Awards 67415-00 55 and 68462-00 56, which are jointly funded by the Professional Staff Congress and The City University of New York.}
\author{Anders Karlsson}\thanks{The second-named author was supported by the Swiss NSF Grants 200020-200400 and 200021-212864, and by the Swedish Research Council Grant 104651320.}
\author{Lejla Smajlovi\'c}
\date{\today}

\begin{document}
\begin{abstract}
We realize the $\mathrm{SU}(3)$ Verlinde sums $V_n(m)$, up to an
explicit factor, as the values at $n$ of the spectral zeta function
of a higher-order Laplace operator on the
triangular discrete torus on $m^2$ vertices. 
For fixed $m$, we express their generating function in terms of the
logarithmic derivative of an associated even spectral polynomial
and express this polynomial as an explicit iterated resultant.
Exploiting permutation symmetry, we prove that this polynomial is
a cube over $\mathbb{Q}$, apart from an explicit quadratic factor
when $3\mid m$.
This factorization yields shorter linear recurrences satisfied by $V_n(m)$ with constant
coefficients.
We also derive Binet-type formulas expressing $V_n(m)$ as finite
linear combinations of powers of rescaled inverse squares of the
roots of the spectral polynomial.
These results provide an efficient algorithm for computing
these Verlinde sums.
Several fully developed examples demonstrating computational efficiency of the method are given, including expressions in terms of Fibonacci and Lucas numbers.
\end{abstract}

\maketitle

\section{Introduction}

Verlinde's formula, originally proposed from conformal field theory, gives the dimension of a certain finite-dimensional vector space associated with a compact Lie group, the genus $g$  of a  Riemann surface, and a positive integral level $k$ \cite{Verlinde}; see also the mathematical introduction to conformal field theory in \cite{Sc08}. Its proofs and interpretation in terms of conformal blocks and generalized theta functions was developed in \cite{BL,TUY}. Different proofs and algebro-geometric formulations were obtained in \cite{Faltings,KNR94,Beauville,Kumar}. Related approaches to Verlinde's formula include stable bundles and stable pairs \cite{Bott,Thaddeus}, the derivation of the $G/G$ model from Chern--Simons theory \cite{BT93}, and a loop-group fixed point formula on the moduli space of flat bundles \cite{AMW01}.

In this paper, as in \cite{Zagier}, we suppress the geometric, algebraic, topological and physical background and study directly the finite trigonometric expressions, which we call \emph{Verlinde sums}.

Early combinatorial and residue theoretic work on the Verlinde formula includes Szenes's verification of the $\mathrm{SU}(2)$ case and his subsequent combinatorial treatment \cite{Sz91,Sz93}. 
Residue and rational trigonometric sum methods were developed further in \cite{Sz98,Sz03}; related residue formulas for vector partitions and Euler--Maclaurin sums appear in \cite{SV03,BrVe}, while wall-crossing and partition-function techniques relevant to this framework appear in \cite{BoVe}. An interested reader may also consult \cite{Ve03} for an exposition of residue formulas for Verlinde sums.
More recently, Verlinde sums were treated as rational trigonometric sums in \cite{LM22}, with their relation to quantization considered in \cite{LM-QR}. Spectral and discrete-analytic methods for related trigonometric sums appear in \cite{Dowker92,CHJSV23a,JKS-cosecant, KM26}.

Dowker remarked in \cite[p. 2642]{Dowker92} that the $\mathrm{SU}(2)$ Verlinde sums are special values of the spectral zeta function of discrete circles and asked whether related sums admit a similar spectral interpretation. We answer this question for $\mathrm{SU}(3)$ by realizing the corresponding Verlinde sums as special values of the spectral zeta function of a higher-order Laplace operator on a triangular discrete torus. Whereas earlier residue and generating-function approaches \cite{Sz93, Zagier} study these sums principally as functions of the level for fixed genus, our spectral-polynomial approach addresses their genus dependence for fixed level.

This spectral realization leads to an explicit and effective algorithm for computing the $\mathrm{SU}(3)$ Verlinde sums for arbitrary fixed level and genus parameter; see Section \ref{sec: alg}. More precisely, we construct the associated spectral polynomial by means of an explicit iterated resultant and use its additional symmetries to reduce the computational complexity. The logarithmic derivative of the reduced polynomial yields a rational generating function, substantially shorter constant coefficient recurrences, and Binet-type formulas expressing the Verlinde sums as finite sums of powers of algebraic numbers. Consequently, the theoretical framework gives a symbolic computational procedure based on polynomial arithmetic, resultants, factorization, and recurrence evaluation. The fully worked examples demonstrate the effectiveness of the method even when the unreduced spectral polynomial has high degree and the resulting Verlinde sums are extremely large.

For example, we prove that the Verlinde sum $V_{g-1}(k+3)$ on  $\text{\rm SU}(3)$ associated to the genus $g>1$ and level $k=2$ can be explicitly expressed in terms of Fibonacci numbers as
$$
V_{g-1}(5)=\left\{
                                                  \begin{array}{ll}
                                                    2\cdot3^{g}5^{\frac{g}{2}}F_{g-1}, &\text{if  } g \text{  is even} \\
                                                    2\cdot3^{g}5^{\frac{g-1}{2}}(F_{g-2} + F_{g}), &\text{if  } g \text{  is odd}
                                                  \end{array}
                                                \right.,
$$
see equation \eqref{vn of 5} in the Appendix. Another example, from the expression \eqref{vn 10} for the Verlinde sum $V_{g-1}(10)$, $n=g-1\geq 1$:
\begin{equation}
V_{n}(10)=
\begin{cases}
\displaystyle
6\cdot 5^{n/2}
\left[(2\cdot 15^n+12^n+2\cdot 60^n)L_n+60^nL_{5n}\right],
& n\ \text{even},\\[4mm]
\displaystyle
6\cdot 5^{(n+1)/2}
\left[(2\cdot 15^n+12^n+2\cdot 60^n)F_n+60^nF_{5n}\right],
& n\ \text{odd},
\end{cases}
\end{equation}
where $L_n$ denotes the $n$th Lucas number.

\subsection{Verlinde sums for \texorpdfstring{$\text{\rm SU}(2)$}{SU(2)}}

The study of Verlinde sums for $\text{\rm SU}(2)$ amounts to understanding the quantities
\begin{equation}
\label{eq:Cm}
C_m(n) := \sum_{j=1}^{m-1} \frac{1}{(\sin(\pi j/m))^{2n}}
\,\,\,\,\,
\text{\rm for $n,m\in\mathbb Z_{>1}$.}
\end{equation}
In \cite{Zagier} there are nine equivalent characterizations of these
numbers. The quantities \eqref{eq:Cm} were computed explicitly, in both untwisted and
twisted forms, by Dowker \cite{Dowker92}, who evaluated the corresponding sums
of powers of cosecants in closed form in connection with Verlinde's formula for
the $\text{\rm SU}(2)$ Wess--Zumino--Witten model; these cosecant sums are
precisely the object recast spectrally in \cite{JKS-cosecant}. The point of view of \cite{JKS-cosecant} adds to the content
of \cite{Dowker92, Zagier} by realizing \eqref{eq:Cm} using the spectrum of the discrete
Laplacian acting on a certain function space on a finite graph.
The graph in question is the discrete circle $\mathbb Z/m\mathbb Z$, and the function space
is the set of functions on $\mathbb Z/m\mathbb Z$ (an additional twist is admissible).
The setting is as follows.

For $r\in\{-(m-1),\dots,m-1\}$, let
\begin{equation}\label{eq:csc-sums}
C_{m,r}(n) := \frac1m\sum_{j=1}^{m-1}\csc^{2n}\!\Big(\frac{j}{m}\pi\Big)e^{2\pi i r j/m}
\end{equation}

As proved in \cite{JKS-cosecant}, the corresponding generating functions
$$
f_{m,r}(s) = \sum_{n\ge0} C_{m,r}(n+1)\,s^n
$$

are for sufficiently small $s$ convergent series and are equal to rational functions which can be expressed explicitly
in terms of Chebyshev polynomials $T_m,U_m$ of the first and second kind.  Specifically,
it is shown in \cite{JKS-cosecant} that
\begin{equation}
\label{eq:fmr}
f_{m,r}(s) = 2\,\frac{U_{m-\ell-1}(1-2s)+U_{\ell-1}(1-2s)}{T_m(1-2s)-1} + \frac{1}{ms},
\end{equation}
where $\ell\in\{0,\dots,m-1\}$ satisfies $\ell\equiv r\pmod m$ and
$U_{-1}\equiv0$. From 
\eqref{eq:fmr} one obtains both
closed forms and a one-step recursion in $m$ for \eqref{eq:Cm}.

\subsection{Verlinde sums for \texorpdfstring{$\text{\rm SU}(3)$}{SU(3)} and spectral graph theory}

Following \cite{Zagier}, we take as the definition of the $\text{\rm SU}(3)$ Verlinde sum $V_n(m)$, for
positive integers $n$ and $m$, to be the series
\begin{equation}
\label{eq:Vn-defn}
V_n(m) = 3^nm^{2n}\!\!\!\sum_{\substack{(j,\ell,k)\in(\mathbb Z/m\mathbb Z \setminus \{0\})^3\\ j+\ell+k\equiv0\bmod m}}\!\!\!\Big(8\sin\frac{\pi j}{m}\sin\frac{\pi\ell}{m}\sin\frac{\pi k}{m}\Big)^{-2n}.
\end{equation}
When comparing our notation to \cite{Zagier},
our $n$ (resp. $m$) corresponds to $g-1$ (resp. $k+3$).  In a slight abuse of
terminology, one can  refer to $n$ as a genus and $m$ as a level, noting here that, in actuality,
$n=g-1$ and $m=k+3$. Therefore, unless otherwise stated, throughout this paper we assume $m\geq 3$.

Consider the standard lattice graph $\mathbb{Z}^{3}$ and its quotient torus $\text{\rm DT}(m) = (\mathbb{Z}/m\mathbb{Z})^{3}$, which is a
 three-dimensional discrete torus with edges in the coordinate directions.  Take a further quotient by the free diagonal action
$$
 (x_1,x_2,x_3)\longmapsto(x_1+1,x_2+1,x_3+1).
$$
The map
$$
 [x_1,x_2,x_3]\longmapsto(x_1-x_3,x_2-x_3)
$$
identifies the quotient with
\[
 G_m=(\mathbb Z/m\mathbb Z)^2.
\]
Writing $x=(x_1,x_2)$, its six neighbors are
\[
 x\pm e_1,\qquad x\pm e_2,\qquad x\pm(e_1+e_2).
\]
Thus $G_m$ is what could be called an example of a  \emph{circulant torus}, or more specifically a six-regular triangular discrete torus on $m^2$ vertices.  The quotient map $\operatorname{DT}(m)\to G_m$ is an $m$-to-one regular covering.
See Figure \ref{fig:cayley-299}.

For $d\in G_m$, let $T_d$ denote the operator
$$
 (T_df)(x):=f(x+d)
$$
and define a directional Laplacian as follows:
$$
\Delta_d =(T_d-I)^*(T_d -I)=2I-T_d-T_{-d}.
$$
The corresponding higher-order Laplace operator is
$$
\Delta_m = \frac 14 \Delta_{e_1}\Delta_{e_2}\Delta_{e_1+e_2}.
$$
\begin{figure}
  \centering
 \includegraphics[width=0.5\textwidth]{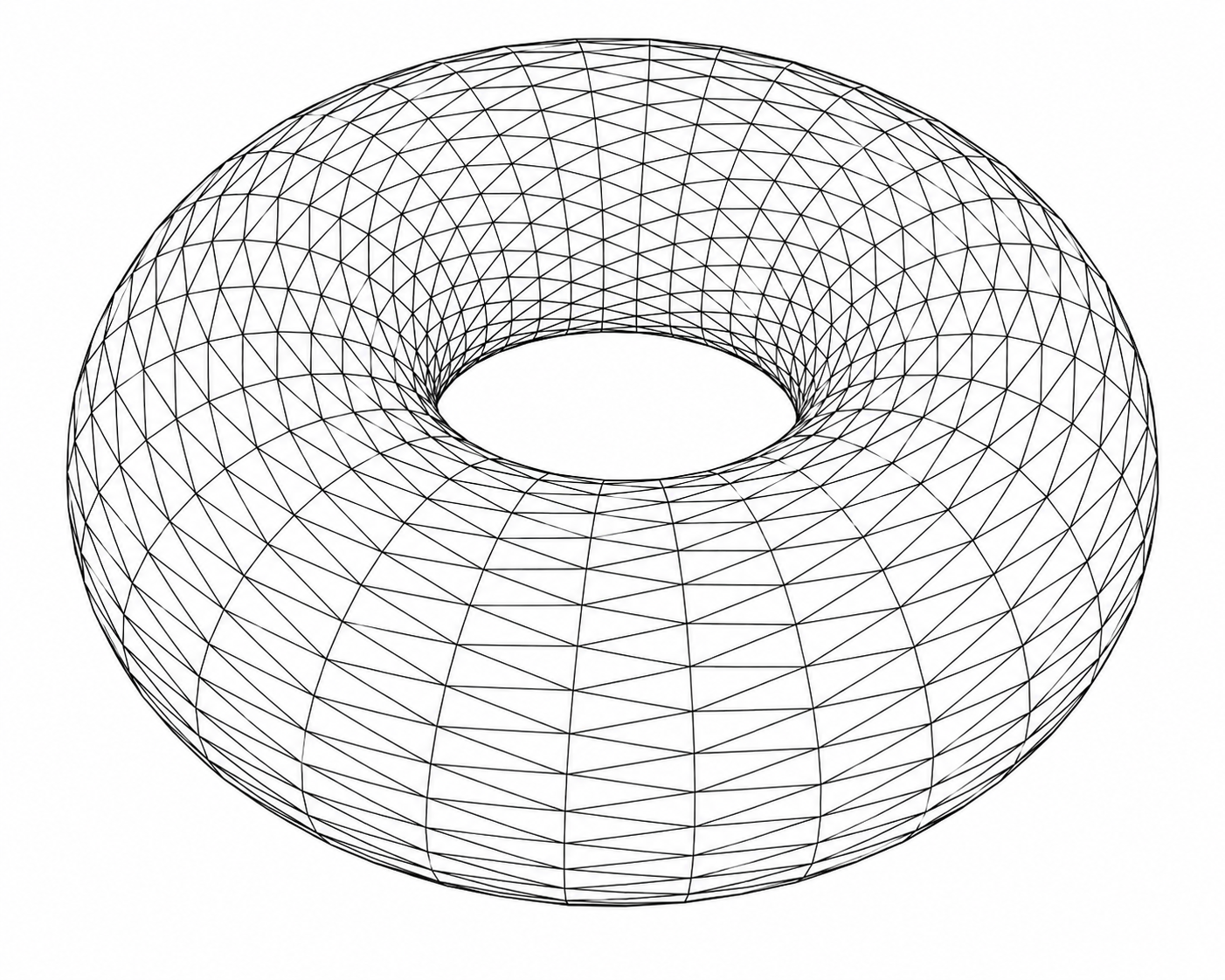}
  \caption{The two dimensional triangular torus $G_m$}.
\label{fig:cayley-299}
\end{figure}
As will be explained below, for $j=(a,b)$, the eigenvalues of $\Delta_m$ are
$$
\lambda_m(j)= 16 \sin^2(\pi a/m)\sin^2(\pi b/m)\sin^2(\pi (a+b)/m).
$$
This gives a spectral zeta interpretation of the Verlinde sums
$$
V_n(m)=\big(3m^2 /4 \big)^n \zeta_{\Delta_m}(n).
$$
Here the spectral zeta function is a sum over the non-zero eigenvalues, more precisely
$$
\zeta_{\Delta_m}(s)= \sum_{j\in G_m \setminus Z} \lambda_m(j)^{-s}
$$
where $Z= \{ j=(a,b):a=0, b=0, \mathrm{ or }\,\, a+b \equiv 0 \}.$ \footnote{The complex function $z\mapsto z^{-s}$ for $z\in\mathbb{C}\setminus\{0\}$ is defined using the principal branch of the logarithm.}


From this construction, we state our first theorem which is the following.

\begin{theorem}
\label{thm:maina}
Let $G_{m}$ be the $2$-dimensional discrete triangular torus on $m^2$ vertices and $\Delta_m$ defined above. Set $c_m:=3m^2/4$. Then
$$
V_n(m)=( 3m^2/4)^n \zeta_{\Delta_m}(n).
$$
Moreover, let $H(w)=\prod_{j\in G_m \setminus Z}(\lambda_m(j)-w)$ be the reduced spectral polynomial of $\Delta_m$, and the rational function
$$
F(w)=-\frac{1}{m^2}\frac{H'(w)}{H(w)}
$$
Then
$$
V_n(m) = 3^n\Big(\frac m2\Big)^{2n}\cdot\frac{m^2}{(n-1)!}\,\partial_w^{\,n-1}F(w)\Big|_{w=0}
$$
and
\begin{align}\label{eq:generating_function}
\sum_{n\ge1} V_n(m)\,\tau^n &= c_m\tau\cdot m^2F(c_m\tau).
\end{align}
as an identity of formal power series in $\tau$ near $\tau=0$. Indeed, the series in
\eqref{eq:generating_function} converges absolutely and uniformly on every closed disc
$\vert \tau \vert\leq R< \frac{1}{c_m}\min_{j\notin Z}\lambda_m(j) $.

\end{theorem}


One can view the last part of Theorem \ref{thm:maina} as the generalization to $\text{\rm SU}(3)$ of
\eqref{eq:fmr} from \cite{JKS-cosecant}.

\subsection{Verlinde sums for \texorpdfstring{$\text{\rm SU}(3)$}{SU(3)} and resultants}

In \cite{JKS-cosecant} it was shown that the resolvent kernel $G(x,y;s)$ for $\mathbb{Z}/m\mathbb{Z}$ when
$x=y$ can be given in terms of Chebyshev polynomials; see 
As a consequence, one was able to obtain linear recursion formulas for the series \eqref{eq:Cm} and
\eqref{eq:csc-sums}.  The following theorem gives the analogous statement for $\textrm{\rm SU}(3)$.
In this case, we replace the Chebyshev polynomials with the resultant $\Res_x(f,g)$ which is constructed from
elementary and explicit polynomials which depend solely on $m$.

For the generating function statement below, set $V_0(m):=(m-1)(m-2)$.

\begin{theorem}
\label{thm:mainb}
For any integer $m \geq 3$ and variables $u, v$ and $w$, let
\begin{equation}\label{eq:resultant_input}
Q(u):=u^m-i^m
\,\,\,\,\,
\text{\rm and}
\,\,\,\,\,
W(u,v,w) := iu^2v^2+u^2v+uv^2-2wuv+u+v-i,
\end{equation}
where $i=\sqrt{-1}$.   Let
\begin{equation}\label{eq:resultant_first}
D(w) = \frac{1}{2^{m^2}w^{3m-2}}\Res_u\Big(Q(u),\ \Res_v\big(Q(v),\,W(u,v,w)\big)\Big).
\end{equation}
Then $D(w)$ is an even, monic polynomial in $w$ of degree $(m-1)(m-2)$.  Furthermore,
for sufficiently small $\tau$ and fixed $m$, the
generating series \eqref{eq:generating_function} for Verlinde sums satisfies the relation that
\begin{equation}
\label{eq:generarating_series}
\sum_{n\ge0} V_n(m)\,\tau^n = (m-1)(m-2) - \sqrt{c_m\tau}\,\frac{D'(\sqrt{c_m\tau})}{D(\sqrt{c_m\tau})}
\,\,\,\,\,
\text{\rm where}
\,\,\,\,\,
c_m:=\frac{3m^2}{4}.
\end{equation}
\end{theorem}

The resultant in \eqref{eq:resultant_first} is sufficiently constructive that it is amenable for
explicit computations.  Throughout this article we will give numerous examples for various $m$.

\subsection{Verlinde sums for \texorpdfstring{$\text{\rm SU}(3)$}{SU(3)} and sums of powers}

If one were to multiply both sides of \eqref{eq:generarating_series} by $D$, one gets a recursive formula in $n$
of length $\textrm{\rm deg}(D)/2 = (m-1)(m-2)/2$ for the Verlinde sums $V_{n}(m)$ for fixed $m$. (The degree is divided by 2, because the generating variable $\tau$ is such that $w^2=c_m\tau$, while $\textrm{\rm deg}(D)$ signifies the degree of $D$ in variable $w$.) The recursive formulas are
obtained by equating coefficients in $\tau$ in the resulting product.  However, because
the right-hand-side of \eqref{eq:generarating_series} is the logarithmic derivative of $D$, the length can be
reduced whenever $D$ has multiple roots, meaning $D$ and $D'$ have common factors.
The following theorem shows that, in general, $D$ is highly divisible,
so then one can expect recursive formulas for $V_{n}(m)$ which are much shorter.

\begin{theorem}
\label{thm:cube_first}
For every $m\ge3$, there is a polynomial $E(w) \in \mathbb{Q}[w]$ such that
\[
D(w) = \begin{cases}E(w)^3, & 3\nmid m,\\ \big(w^2-\tfrac{27}{4}\big)E(w)^3, & 3\mid m.\end{cases}
\]
As such, $\deg E = (m-1)(m-2)/3$ if $3\nmid m$ and $\deg E = ((m-1)(m-2)-2)/3$ if $3\mid m$.
\end{theorem}

From Theorem \ref{thm:cube_first}, we have that
$$
\frac{D'}{D} = \frac{3E'}{E} + \frac{2w}{w^{2}-27/4} \mathbf 1_{3\mid m}
$$
where $\mathbf 1_{3\mid m} = 1$ if $3$ divides $m$ and zero otherwise, and the prime denotes the derivative in $w$. Therefore,
Theorem \ref{thm:cube_first} implies that the length of the recursive sequence is at
most $\frac{\deg E}{2} + \mathbf 1_{3\mid m}$.
In most examples, the values attained by $V_{n}(m)$ are very large, so the reduction
in the length of the recursive formula is significant when obtaining numerical evaluations.

There is one additional noteworthy aspect of Theorem \ref{thm:cube_first}, namely that in the generating series
\eqref{eq:generarating_series}, the right-hand-side has simple poles.
Therefore, from the elementary expansion
$$
\frac{bx}{x-a} = -b\sum\limits_{n=1}^{\infty}(x/a)^{n}
\,\,\,\,\,
\textrm{\rm for constants $a, \, b$ and $a\neq 0$ and  $|x|<|a|$},
$$
we get the following expression for Verlinde sums in terms of the roots of $E$.

\begin{theorem}
\label{thm:algebraic_sum} With notation as above, let $\mathcal{E}_{m} = \{a\}$ be the set of all distinct roots of $E$, each with multiplicity
$\mu(a)$.  Then, for any $n \geq 1$, we have that
\begin{equation}\label{eq:algebraic_sum_formula1}
V_n(m) \;=\; \mathbf 1_{3\mid m}\cdot 2\Big(\frac{m}{3}\Big)^{2n} \;+\; 3\sum_{a \in \mathcal{E}_m}
\mu(a)\left(\frac{3m^2}{4a^2}\right)^{\!n}.
\end{equation}
\end{theorem}

Theorem \ref{thm:algebraic_sum} is the analogue of the classical Binet formula for linear recurrence sequence.
In effect, the factor of $2$ in \eqref{eq:algebraic_sum_formula1} is because of the change of variables
from $w$ to $\tau$, and the additional factor of $3$ in the sum over roots of $E$ is from Theorem \ref{thm:cube_first}.
Any additional appearance of repeated roots of $E$ manifests itself through the factor $\mu(a)$. In the case
$m=10$ which we give in the Appendix below, the resultant $D$ has degree $72$, yet we obtain a Binet formula involving $4$ pairs of algebraic
integers in $\mathbb{Q}(\sqrt{5})$.

Alternatively, let $\widehat E(x)=E(w)|_{w^2=x}$ and define its reciprocal polynomial by
\[
\widehat E^\vee(y)=y^{\deg\widehat E}\widehat E(1/y).
\]
Let $\mathcal E^\vee$ be the set of distinct roots of $\widehat E^\vee$, with multiplicities $\mu(\alpha)$.  Then
\begin{equation}\label{eq:algebraic_sum_formula2}
V_n(m)=\mathbf1_{3\mid m}\,2\left(\frac m3\right)^{2n}
+6\cdot3^n\left(\frac m2\right)^{2n}
\sum_{\alpha\in\mathcal E^\vee}\mu(\alpha)\alpha^n.
\end{equation}


\subsection{An algorithm for Verlinde sums for \texorpdfstring{$\text{\rm SU}(3)$}{SU(3)}} \label{sec: alg}

When combining the results above, we obtain the following effective algorithm to explicitly compute $V_n(m)$.

\medskip
\noindent\textbf{Algorithm (Computation of $V_n(m)$).}
\begin{quote}
\textsc{Input:} Integers $m\ge3$ and $n\ge1$.

\begin{enumerate}
\item[\textup{1.}] Let $Q(u):=u^m-i^m$ and $W(u,v,w) := iu^2v^2+u^2v+uv^2-2wuv+u+v-i$,
and compute the iterated resultant
$$
D(w) \;=\; \frac{1}{2^{m^2}w^{3m-2}}\,\mathrm{Res}_u\Big(Q(u),\ \mathrm{Res}_v\big(Q(v),\,W(u,v,w)\big)\Big).
$$
\item[\textup{2.}] Extract the reduced polynomial $E(w)$
and let $\widehat E(x):=E(w)|_{w^2=x}$.
\item[\textup{3.}] Let
$E^\vee(y):=y^{\deg\widehat E}\,\widehat E(1/y)$.  Let $\mathcal{E}^{\vee}$ be the set of
distinct roots of $E^{\vee}$
where each root $\alpha$ has multiplicity $\mu(\alpha)$.  Set
$$
p_n \;:=\; \sum_{\alpha \in \mathcal{E}^{\vee}} \mu(\alpha) \alpha^n.
$$
\end{enumerate}

\textsc{Output:} The Verlinde sum $V_{n}(m)$ is given by
\begin{equation}\label{eq:Verlinde_power_sum}
V_n(m) \;=\; \mathbf{1}_{3\mid m}\cdot 2\Big(\frac{m}{3}\Big)^{2n} \;+\; 6\cdot 3^{n}\Big(\frac{m}{2}\Big)^{2n}\,p_n.
\end{equation}
\end{quote}
\medskip

\subsection{Organization of this paper} In section \ref{sec:circulant-torus} we establish notation and
recall background information.  In particular, we discuss the spectral theory
on the circulant torus, which is needed for Theorem \ref{thm:mainb}, and basic facts for the resultant and linear recurrences, which
is needed for the other results stated above.  In section \ref{sec:Proof_of_Maina} we prove Theorem
\ref{thm:maina}.  In Section \ref{sec:first-construction} we prove the first resultant construction on $G_{m}$,
present a few examples, and discuss its inefficiency.   Our second
resultant construction is given in Section \ref{sec:second-construction}, namely formula \eqref{eq:resultant_first}.
Further examples are given which demonstrate the improvement over the resultant
construction from Section \ref{sec:first-construction}.  In Section \ref{sec:generating-series} we complete the proof of Theorem \ref{thm:mainb} and derive, from the rational
generating function, the linear recurrences for $V_n(m)$; these use the function $D$ from \eqref{eq:resultant_first}.  In Section \ref{sec:cubesection} we prove Theorem \ref{thm:cube_first} which allows one
to repeat the analysis from \ref{sec:first-construction} and prove recursion formulas which are much shorter.
Many examples are given which compare the different recursion formulas given.
Also, in Section \ref{sec:cubesection}, we prove Theorem \ref{thm:algebraic_sum}, and  give some comments as to why one could expect the length of the recursion
formulas to shorten even further.  In Appendix we give several examples  with $m=5,6,7,9,10,14$ in their entirety.

\section{Resultants, linear recurrences and circulant tori}
\label{sec:circulant-torus}

\subsection{The resultant}
\label{subsec:sylvester}
Let $f(t)=\sum_{i=0}^m f_it^i$ and $g(t)=\sum_{j=0}^n g_jt^j$ be non-zero
polynomials over a commutative ring of degrees $m,n$ respectively.
We assume $f_{m} \neq 0$ and $g_{n} \neq 0$.
The Sylvester matrix $\mathrm{Syl}(f,g)$ is the
$(m+n)\times(m+n)$ matrix whose first $n$ rows are successive shifts of
the coefficient vector of $f$ and whose last $m$ rows are successive
shifts of the coefficient vector of $g$:
\begin{equation}\label{eq:sylvester_matrix}
\text{Syl}(f, g) =
\begin{pmatrix}
f_m    & f_{m-1} & \cdots & f_0    & 0       & \cdots & 0 \\
0      & f_m     & f_{m-1} & \cdots & f_0     & \cdots & 0 \\
\vdots &         & \ddots  & \ddots &         & \ddots & \vdots \\
0      & \cdots  & 0       & f_m    & f_{m-1} & \cdots & f_0 \\
g_n    & g_{n-1} & \cdots & g_0    & 0       & \cdots & 0 \\
0      & g_n     & g_{n-1} & \cdots & g_0     & \cdots & 0 \\
\vdots &         & \ddots  & \ddots &         & \ddots & \vdots \\
0      & \cdots  & 0       & g_n    & g_{n-1} & \cdots & g_0
\end{pmatrix}
\end{equation}
 The resultant of $f$ and $g$, denoted by $\Res(f,g)$, is the determinant of $\mathrm{Syl}(f,g)$, written as
$$
\Res(f,g) := \det\mathrm{Syl}(f,g).
$$
The resultant $\Res(f,g)$ is a polynomial with integer coefficients in the
coefficients of $f$ and $g$.  The following lemma summarizes well-known
properties of $\Res(f,g)$, which we present here for the convenience
of the reader.

\begin{lemma}
\label{prop:resultant-basic}
Assume the notation as above and that the commutative ring is a field.
\begin{enumerate}[label=(\roman*)]
\item We have that $\Res(f,g)=0$ if and only if
positive degree polynomials $f$ and $g$ have a nonconstant common
factor.
\item  If $f$ factors completely into linear factors, perhaps in some splitting field, and has roots
$\alpha_1,\dots,\alpha_m$, counted with multiplicity, then
\begin{equation}
\label{eq:product-formula}
\Res(f,g) = f_m^n\prod_{i=1}^m g(\alpha_i).
\end{equation}
In particular, if $f$ is monic, then $\Res(f,g)=\prod_{i=1}^m g(\alpha_i)$.
Symmetrically, if $g$ factors completely into linear factors and has roots $\beta_1,\dots,\beta_n$, then
$\Res(f,g)=(-1)^{mn}g_n^m\prod_{j=1}^n f(\beta_j)$.
\end{enumerate}
\end{lemma}

Classically, resultants have been used to show, for example, that the sum of two algebraic integers
is an algebraic integer.  To do so, one forms the polynomial in $y$ obtained by computing
the resultant in $x$ of $\tilde{f}(x) = f(x)$ and $\tilde{g}(x) = g(y-x)$, which we write
as $\Res_{x}(\tilde{f},\tilde{g})(y)$.  By \eqref{eq:product-formula}, $\Res_{x}(\tilde{f},\tilde{g})(y)$
is zero if and only if $y = \alpha + \beta$ where $\alpha $ is a root of $f$ and $\beta$ is
a root of $g$.  If both $\tilde{f}$ and $\tilde{g}$ are monic, then so is $\Res_{x}(\tilde{f},\tilde{g})(y)$,
thus showing that $\alpha + \beta$ is an algebraic integer if both $\alpha$ and $\beta$ are algebraic integers.

We will use a similar consideration in this article in our proof of Theorem \ref{thm:mainb}.

\subsection{Closed form solutions of linear recurrences}
Assume the sequence $\{a_{k}\}$ for $k \geq 0$ satisfies an $n$ step linear homogeneous recurrence relation with constant
coefficients.  In other words, for $k \ge n$ we have that
\begin{equation}
\label{eq:app-recurrence}
a_k = c_1a_{k-1}+c_2a_{k-2}+\cdots+c_na_{k-n},
\end{equation}
with $n$ prescribed initial terms $\{a_0,a_1,\dots,a_{n-1}\}$.
Substituting $a_h=r^h$ for all $h$ into \eqref{eq:app-recurrence} yields the
degree $n$ characteristic equation
\begin{equation}
\label{eq:app-char}
f(r):= r^n - c_1r^{n-1} - c_2r^{n-2} - \cdots - c_{n-1}r - c_n = 0.
\end{equation}
Let $\{r_{j}\}_{j=1}^n$ denote the roots of the characteristic polynomial \eqref{eq:app-char}, computed
in a splitting field and counted with multiplicity.
If all roots $\{r_1,r_2,\dots,r_n\}$ are distinct, then the explicit closed
form for $a_{k}$ is given as
\begin{equation}
\label{eq:app-distinct}
a_k = C_1(r_1)^k+C_2(r_2)^k+\cdots+C_n(r_n)^k
\end{equation}
where the constants $C=[C_1,\dots,C_n]^T$ are obtained by solving the
$n\times n$ Vandermonde linear system
\[
\begin{pmatrix}1&1&\cdots&1\\r_1&r_2&\cdots&r_n\\\vdots&\vdots&&\vdots\\r_1^{n-1}&r_2^{n-1}&\cdots&r_n^{n-1}\end{pmatrix}
\begin{pmatrix}C_1\\C_2\\\vdots\\C_n\end{pmatrix} = \begin{pmatrix}a_0\\a_1\\\vdots\\a_{n-1}\end{pmatrix}.
\]
If a root $r_j$ repeats with multiplicity $\mu$, its
contribution to the solution of $a_{k}$ is not $C_{j}r_{j}^k$ but rather is
a polynomial in $k$ of degree $\mu-1$, namely
\begin{equation}
\label{eq:app-repeated}
\big(B_0+B_1k+B_2k^2+\cdots+B_{\mu-1}k^{\mu-1}\big)(r_j)^k.
\end{equation}
One no longer uses the Vandermonde linear system to determine the coefficients $\{B_{j}, C_{j}\}$,
but instead forms the obvious system of equations which determines the coefficients from the first
$n$ terms in the recursive sequence.

Define the ordinary generating function of the sequence $\{a_{k}\}$ as the formal infinite power
series $A(x)=\sum_{k=0}^\infty a_kx^k$.  It can be shown that $A$ is a  rational function,
meaning for $\vert x \vert$ sufficiently small, we have that
\begin{equation}
\label{eq:app-genfunc}
A(x) = \frac{P(x)}{1-c_1x-c_2x^2-\cdots-c_nx^n}.
\end{equation}
The denominator is the reciprocal characteristic polynomial, meaning $r^{n}f(1/r)$, and the
numerator $P(x)$ is a polynomial of degree $\le n-1$ encoding the
initial conditions of the recursive sequence.  More precisely, $P(x)$ is given by
\[
P(x) = \sum_{i=0}^{n-1}\Big(a_i-\sum_{j=1}^ic_ja_{i-j}\Big)x^i.
\]

If we factor the denominator into linear terms, then we can expand the denominator using a
geometric series for each term and obtain the closed formula \eqref{eq:app-distinct}
for each $a_{k}$ as sum of $k$-th powers of the roots
of the characteristic polynomial.

For the questions we consider in this article, the characteristic polynomial has distinct roots, so
\eqref{eq:app-repeated} does not play a role.

\subsection{Discrete and circulant tori}
For an integer $m\ge1$, the discrete circle $\mathbb Z/m\mathbb Z$ is the
graph on $m$ vertices with edges joining nearest neighbors, a $2$-regular
graph. For $M=(m_1,\dots,m_k)$, the discrete torus $\mathrm{DT}(M)$ is
the Cayley graph of $\prod_j\mathbb Z/m_j\mathbb Z$ with respect to the
standard basis vectors. Throughout, we take $k=3$, $m_1=m_2=m_3=m$, and
write $\mathrm{DT}(m)$ for the resulting $6$-regular graph on $m^3$
vertices.

Set $\mathrm{DT}(m)=(\mathbb Z/m\mathbb Z)^3$.  The vertex
$(x_1,x_2,x_3)$  is adjacent to $(x_1',x_2',x_3')$ if and only if
the two ordered triples differ by $\pm1$ in exactly one coordinate. Define the
quotient graph
$$
G_m := \mathrm{DT}(m)/\!\sim
\,\,\,\,\,
\text{\rm where}
\,\,\,\,\, (x_1,x_2,x_3)\sim(x_1+1,x_2+1,x_3+1).
$$
The cyclic group $Z/m\mathbb Z$, which is
generated by the diagonal shift which defines $G_{m}$,
acts freely on $\mathrm{DT}(m)$ with
quotient $G_m$.  Let us write $(x_1,x_2)$ for the coset of $(x_1,x_2,0)$, so
that the vertex $(x_1,x_2,x_3)\in\mathrm{DT}(m)$ lies in the coset
$(x_1-x_3,x_2-x_3)$. By tracking the edges to $(x_1,x_2,0)$ under the
quotient map, we can show that $(x_{1},x_{2})$ has six edges in $G_{m}$.
First, we have the four edges to $(x_1\pm1,x_2,0)$ and $(x_1,x_2\pm1,0)$ in $\mathrm{DT}(m)$
which descend to the edges $(x_1,x_2)\sim(x_1\pm1,x_2)$ and
$(x_1,x_2)\sim(x_1,x_2\pm1)$ in $G_m$.  Second, we have the two edges to
$(x_1,x_2,\pm1)$ in $\mathrm{DT}(m)$ which descend to edges $(x_1,x_2)\sim(x_1\mp 1,x_2\mp 1)$, respectively,
in $G_{m}$.
Thus $G_m$ is a $6$-regular graph on $m^2$ vertices
with $3m^2$ edges, and $\mathrm{DT}(m)\to G_m$ is an $m$-to-$1$ regular
covering. We call $G_m$ a \emph{circulant torus}, see Fig. \ref{fig:cayley-299}.

\subsection{Eigendata on the discrete torus}

For $d\in G_m$, let $T_d$ denote the operator
$$
 (T_df)(x):=f(x+d).
$$
These operators commute and are unitary, $T^*_d=T_{-d}$.

With this notation one can define a directional Laplacian as follows:
$$
\Delta_d =(T_d-I)^*(T_d -I)=2I-T_d-T_{-d}.
$$

Define a corresponding higher-order Laplace operator:
$$
\Delta_m = \frac 14 \Delta_{e_1}\Delta_{e_2}\Delta_{e_1+e_2}.
$$
This is related to Weyl's denominator formula for $\text{\rm SU}(3)$.

An alternative to $\Delta_m$ is the following. Note the elementary identity for any real $\alpha$ and $\beta$,
$$
4\sin\alpha\sin\beta\sin(\alpha+\beta) = \sin2\alpha+\sin2\beta-\sin2(\alpha+\beta).
$$
This has the operator manifestation as
$$
A_ m= \tfrac{i}{6}(T_{e_1}-T_{-e_1}+T_{e_2}-T_{-e_2}+T_{-e_1-e_2}-T_{e_1+e_2}).
$$
The triangles can be oriented so that $e_1$, $e_2$, and $-e_1 -e_2$ provide the corresponding edge orientation. The above operator corresponds to the oriented graph adjacency operator. Note that
$$
\Delta_m = 9A_m^*A_m.
$$

For $j=(a,b)\in (\mathbb Z/m\mathbb Z)^2$, define
$$
 \chi_{j}(x_1,x_2):=
 \frac1m\exp\left(2\pi i(ax_1+bx_2)/m\right).
$$
These characters form an orthonormal basis of
$\ell^2(G_m)$.  They also serve as eigenfunctions to the basic operators
$$
 T_{(u,v)}\chi_{j}
 =\exp \left(2\pi i(au+bv)/m\right)\chi_{j}
$$
and hence also of  $\Delta_m$.



The spectral zeta function is a sum over the non-zero eigenvalues, more precisely
$$
\zeta_{\Delta_m}(s)= \sum_{j\in G_m \setminus Z} \lambda_m (j)^{-s}
$$
where $Z= \{ j=(a,b):a=0, b=0, \mathrm{ or }\,\, a+b \equiv 0 \}.$ The number of terms in this sum is
\[
M = m^2-(3m-2) = (m-1)(m-2).
\]

The reduced spectral polynomial is
$$
H(w)=\prod_{j\in G_m \setminus Z}(\lambda_m(j)-w)
$$
Take its logarithmic derivative and define
$$
F(w)=-\frac{1}{m^2}\frac{H'(w)}{H(w)}.
$$
This is clearly a rational function.


Let
\[
 r_{a,b}
 =-\sin(2\pi a/m)-\sin (2\pi b/m)+\sin(2(a+b)\pi /m).
\]
In view of the trigonometric identity just mentioned,
$$
 r_{a,b}^2=\lambda_m((a,b))
 =16\sin^2\frac{\pi a}{m}\sin^2\frac{\pi b}{m}
 \sin^2\frac{\pi(a+b)}{m}.
 $$

We also define
\[
m^2 G(w)=m^{2} G(x,x;w) := \sum_{j\in G_m\setminus Z}\frac{1}{w-r_j},
\]
which can be viewed as a resolvent kernel for any $x\in G_m$.

Define another spectral determinant polynomial by
\begin{equation}\label{eq:det_polynomial}
D(w) := \prod_{j\in G_m \setminus Z}(w-r_j).
\end{equation}
Then $D(w)$ is
a monic polynomial of degree $(m-1)(m-2)$ in $w$ and whose roots counted
with multiplicity are
$r_j$.
The involution $j\mapsto-j$ sends $r_j$ to $-r_j$, so $D$ is even. Therefore it can be written $D(w)=\hat {D}(w^2)$. Then $H(w)=\hat{D}(w)^2$.

Note that
\[
m^{2}G(w) = \frac{D'(w)}{D(w)} = \frac{d}{dw}\log D(w)
\]
for any vertex $x \in G_{m}$. Moreover,
\[
F(w):= F(x,x;w)=-\frac1{2\sqrt w}
 \bigl(G(x,x;\sqrt w)-G(x,x;-\sqrt w)\bigr).
\]

We will use the following notation. Let us set
\begin{equation}
\label{eq:beta-Lambda}
\Lambda_{m}(j) := \frac13\sum_{\ell=1}^3\cos\frac{2\pi(j_\ell+m/4)}{m}
\end{equation}
Then we have that
\begin{equation}\label{eq. 3Lambda explicit}
3\Lambda_{m}(j) = (-1)^\ell4\sin\frac{\pi j_1}{m}\sin\frac{\pi j_2}{m}\sin\frac{\pi j_3}{m}
\,\,\,\,\,
\text{\rm whenever $j_1+j_2+j_3=\ell m$.}
\end{equation}
For $(j_1,j_2,j_3)\in(\mathbb Z/m\mathbb Z \setminus \{0\})^3$ admissible values of $\ell$ for which $j_1+j_2+j_3=\ell m$ are $\ell\in\{1,2\}$.
Let
$$
A(m):=\{j\in(\mathbb Z/m\mathbb Z)^3 : j_1+j_2+j_3\equiv0\}
\,\,\,\,\,
\text{\rm and}
\,\,\,\,\,
A_0(m):=\{j\in A(m):\Lambda_{m}(j)=0\}.
$$
There are $m^{3}$ ordered triples $(j,\ell, k) \bmod m$, of which $m^2$ are such that
$j+\ell+k\equiv0\bmod m$, so $A(m)$ has $m^{2}$ elements.
Among the $m^2$ ordered triples
$(j,\ell,k)\bmod m$ with
$j+\ell+k\equiv0$, there are $3m-2$ triples that have at least one coordinate congruent to zero modulo $m$,
so $A_0(m)$ has $3m-2$ elements.


Define $i_m:G_m \rightarrow A(m)$ via $i_m(a,b)=(a,b,-a-b)$.
Note that $r_{a,b}=3\Lambda_m (i_m(a,b))$ and $\lambda_m((a,b))=r_{a,b}^2$.

\section{Proof and examples of Theorem \ref{thm:maina}}
\label{sec:Proof_of_Maina}

\subsection{Proof of Theorem \ref{thm:maina}}

\begin{proof}
Fix \(m\geq 3\) and set \(c_m=3m^2/4\).
For \(j=(a,b)\in G_m\) and \(d=(u,v)\in G_m\), the Fourier
characters defined above satisfy
\[
T_d\chi_j
=
\exp\!\left(\frac{2\pi i(au+bv)}{m}\right)\chi_j.
\]
Consequently,
\begin{align*}
\Delta_d\chi_j
&=
\left(
2-\exp\left(\frac{2\pi i(au+bv)}{m}\right)
 -\exp\left(-\frac{2\pi i(au+bv)}{m}\right)
\right)\chi_j \\
&=
4\sin^2\left(\frac{\pi(au+bv)}{m}\right)\chi_j.
\end{align*}
Since
\[
\Delta_m=\frac14\Delta_{e_1}\Delta_{e_2}\Delta_{e_1+e_2},
\]
we obtain
\[
\Delta_m\chi_j=\lambda_m(j)\chi_j,
\qquad
\lambda_m(j)
=
16\sin^2\left(\frac{\pi a}{m}\right)
  \sin^2\left(\frac{\pi b}{m}\right)
  \sin^2\left(\frac{\pi(a+b)}{m}\right).
\]
The characters form an orthonormal basis of \(\ell^2(G_m)\),
so these are all the eigenvalues, counted with multiplicity.
Moreover, \(\lambda_m(j)=0\) precisely when
\(a=0\), \(b=0\), or \(a+b=0\) in \(\mathbb Z/m\mathbb Z\).
Thus the nonzero eigenvalues are indexed by \(G_m\setminus Z\),
and they are strictly positive.

The map
\[
(a,b)\longmapsto(a,b,-a-b)
\]
is a bijection from \(G_m\setminus Z\) onto the set of triples
occurring in the definition of \(V_n(m)\).
For the corresponding triple \((a,b,k)\), we have
\(k\equiv-a-b\pmod m\), and hence
\[
\sin^2\left(\frac{\pi k}{m}\right)
=
\sin^2\left(\frac{\pi(a+b)}{m}\right).
\]
It follows that
\[
\left(
8\sin\left(\frac{\pi a}{m}\right)
 \sin\left(\frac{\pi b}{m}\right)
 \sin\left(\frac{\pi k}{m}\right)
\right)^2
=
4\lambda_m((a,b)).
\]
Therefore, for every integer \(n\geq 1\),
\begin{align*}
V_n(m)
&=
3^n m^{2n}
\sum_{j\in G_m\setminus Z}
\bigl(4\lambda_m(j)\bigr)^{-n} \\
&=
\left(\frac{3m^2}{4}\right)^n
\sum_{j\in G_m\setminus Z}\lambda_m(j)^{-n} \\
&=
c_m^n\zeta_{\Delta_m}(n).
\end{align*}
This proves the first identity.

Next, logarithmic differentiation of the finite product
\[
H(w)=\prod_{j\in G_m\setminus Z}\bigl(\lambda_m(j)-w\bigr)
\]
gives
\[
\frac{H'(w)}{H(w)}
=
-\sum_{j\in G_m\setminus Z}
\frac{1}{\lambda_m(j)-w}.
\]
By the definition of \(F\), we therefore have
\[
m^2F(w)
=
-\frac{H'(w)}{H(w)}
=
\sum_{j\in G_m\setminus Z}
\frac{1}{\lambda_m(j)-w}.
\]
In particular, \(F\) is rational and is holomorphic near \(w=0\),
since \(H(0)\neq 0\).

Set
\[
\lambda_*:=\min_{j\in G_m\setminus Z}\lambda_m(j)>0.
\]
For \(|w|<\lambda_*\), each summand has the geometric expansion
\[
\frac{1}{\lambda_m(j)-w}
=
\sum_{k=0}^{\infty}
\frac{w^k}{\lambda_m(j)^{k+1}}.
\]
These expansions converge absolutely and uniformly on every
closed disk \(|w|\leq\rho<\lambda_*\).
Since the sum over \(j\) is finite, we may interchange the
summations to obtain
\begin{align*}
m^2F(w)
&=
\sum_{k=0}^{\infty}
w^k
\sum_{j\in G_m\setminus Z}\lambda_m(j)^{-k-1} \\
&=
\sum_{n=1}^{\infty}
\zeta_{\Delta_m}(n)w^{n-1} \\
&=
\sum_{n=1}^{\infty}
c_m^{-n}V_n(m)w^{n-1}.
\end{align*}
Taking the \((n-1)\)-st derivative at \(w=0\) yields
\[
m^2F^{(n-1)}(0)
=
(n-1)!\,c_m^{-n}V_n(m).
\]
Since \(c_m^n=3^n(m/2)^{2n}\), this is equivalent to
\[
V_n(m)
=
3^n\left(\frac{m}{2}\right)^{2n}
\frac{m^2}{(n-1)!}
\left.
\frac{\partial^{n-1}}{\partial w^{n-1}}F(w)
\right|_{w=0},
\]
as claimed.

Finally, multiplying the power-series expansion of \(m^2F(w)\)
by \(w\) and substituting \(w=c_m\tau\), we obtain
\[
c_m\tau\,m^2F(c_m\tau)
=
\sum_{n=1}^{\infty}
c_m^{-n}V_n(m)(c_m\tau)^n
=
\sum_{n=1}^{\infty}V_n(m)\tau^n.
\]
This identity holds analytically for
\[
|\tau|<\frac{\lambda_*}{c_m},
\]
and the series converges absolutely and uniformly on every
closed disk \(|\tau|\leq R<\lambda_*/c_m\).

The same geometric expansions are valid as formal power series,
because every factor \(\lambda_m(j)-w\) has a nonzero constant
term. Hence the generating-function identity also holds as an
identity of formal power series in \(\tau\).
\end{proof}

\subsection{Examples}
\label{sec:examples}

The examples in this section are built using the precise information for the
spectrum of $G_{m}$.  Admittedly, these computations are somewhat inefficient.  We are presenting them
here as a benchmark for comparison with later examples.

\begin{example}[$m=3$]
\label{ex:m3}
The two nonzero eigenvalues are $r = \pm\frac{3\sqrt3}{2}$, each of multiplicity $1$. Then
\[
m^2 G(w) = \frac{1}{w - \tfrac{3\sqrt3}{2}} + \frac{1}{w+\tfrac{3\sqrt3}{2}}
= \frac{2w}{w^2 - \tfrac{27}{4}} = \frac{8w}{4w^2-27}.
\]
Since $m^2 = 9$, we get that
\[
G(w) = \frac{8w}{9(4w^2-27)}
\]
and
\[
F(w) = -\frac{G(\sqrt w)}{\sqrt w} = -\frac{8}{9(4w-27)} = \frac{8}{9(27-4w)}.
\]
When expanding as a geometric series about $w=0$, we have that
\[
F(w) = \frac{8}{243}\sum_{k\ge0}\Bigl(\frac{4w}{27}\Bigr)^k
= \sum_{k\ge0} \frac{8}{243}\Bigl(\frac{4}{27}\Bigr)^k w^k,
\]
so $\partial_w^{n-1}F(w)\big|_{w=0} = (n-1)!\,\frac{8}{243}\bigl(\tfrac{4}{27}\bigr)^{n-1}$, and
\[
V_n(3) = 3^n\Bigl(\frac{3}{2}\Bigr)^{2n}\cdot 9\cdot \frac{8}{243}\Bigl(\frac{4}{27}\Bigr)^{n-1} = 2
\]
for every $n \geq 1$.  Thus, $V_n(3) = 2$ for all $n \ge 1$.  Separately, since $c_3 = 3(3/2)^2 = 27/4$, we get that
\[
\sum_{n\ge1} V_n(3)\,\tau^n = c_3\tau\cdot 9\cdot F(x,x;c_3\tau)
= \frac{27}{4}\tau\cdot 9\cdot \frac{8}{9\bigl(27 - 27\tau\bigr)}
= \frac{2\tau}{1-\tau} = 2\sum_{n\ge1}\tau^n,
\]
again giving $V_n(3)=2$ for every $n\ge1$.
\end{example}

\begin{example}[$m=5$]
\label{ex:m5}
The twelve nonzero eigenvalues fall into four classes of values each with multiplicity $3$.  Specifically,
the eigenvalues are of the form $r_j = \pm a_1$ and $r_j=\pm a_2$ where
\[
a_1 = \sqrt{\frac{25+5\sqrt5}{8}}, \qquad a_2 = \sqrt{\frac{25-5\sqrt5}{8}}.
\]
Then
\[
m^2 G(w) = \frac{3\cdot 2w}{w^2-a_1^2} + \frac{3\cdot 2w}{w^2-a_2^2}
= \frac{6w}{w^2-\tfrac{25+5\sqrt5}{8}} + \frac{6w}{w^2-\tfrac{25-5\sqrt5}{8}}.
\]
When combining over a common denominator, using $a_1^2+a_2^2 = \tfrac{25}{4}$,
$a_1^2 a_2^2 = \tfrac{125}{16}$, and $m^2=25$ gives
\[
G(w) = \frac{192w^3-600w}{400w^4-2500w^2+3125}
\]
and then
\[
F(w) = -\frac{G(\sqrt w)}{\sqrt w}
= \frac{600-192w}{400w^2-2500w+3125}.
\]
When expanding about $w=0$, we get that
\[
V_1(5) = 90,\qquad V_2(5) = 810, \qquad V_3(5)=8100,\qquad V_4(5) = 85050.
\]
With $c_5 = 3\cdot 25/4 = 75/4$,
\[
\sum_{n\ge1} V_n(5)\,\tau^n = c_5\tau\cdot 25\cdot F(c_5\tau)
= \frac{90\tau(1-6\tau)}{45\tau^2-15\tau+1}.
\]
Now when expanding in $\tau$ reproduces $V_1(5),\dots,V_4(5)$ above, and the two poles of the denominator
$45\tau^2-15\tau+1=0$ are $\tau = a_1^2/c_5$ and $\tau=a_2^2/c_5$.
\end{example}

\begin{example}[$m = 6$]
\label{ex:m6}
By direct evaluation, we have that $\Lambda_{6}(j) = \frac13\sum_{\ell=1}^3\cos\bigl(2\pi(j_\ell+\tfrac32)/6\bigr)$
over the $20 = (6-1)(6-2)$ triples of $A(6)\setminus A_0(6)$ produces exactly three
classes, namely
\[
r = \pm\frac{3\sqrt3}{2}\ (\text{mult.}\ 1), \qquad r=\pm\sqrt3\ (\text{mult.}\ 6), \qquad
r = \pm\frac{\sqrt3}{2}\ (\text{mult.}\ 3),
\]
with total multiplicity being $1+1+6+6+3+3=20 = (6-1)(6-2)$ as required.  Then
\[
m^2 G(w) = \frac{2w}{w^2-\tfrac{27}{4}} + \frac{12w}{w^2-3} + \frac{6w}{w^2-\tfrac34},
\]
which, combined over a common denominator, using $m^2=36$ becomes
\[
G(w) = \frac{80w^5-624w^3+747w}{144w^6-1512w^4+3969w^2-2187}
\]
and then
\[
F(w) = -\frac{G(x,x;\sqrt w)}{\sqrt w}
= \frac{-80w^2+624w-747}{144w^3-1512w^2+3969w-2187}.
\]
Expanding $F(w)$ about $w=0$ and applying
$V_n(m) = 3^n(m/2)^{2n}\cdot m^2/(n-1)!\cdot \partial_w^{n-1}F(x,x;w)\big|_{w=0}$ term by term gives
\[
V_1(6)=332,\quad V_2(6)=8780,\quad V_3(6)=288812,\quad V_4(6)=10156940.
\]
With $c_6 = 3\cdot 6^2/4 = 27$,
\[
\sum_{n\ge1} V_n(6)\,\tau^n = c_6\tau\cdot 36\cdot F(x,x;c_6\tau)
= \frac{4\tau\bigl(6480\tau^2-1872\tau+83\bigr)}{(1-4\tau)(1-9\tau)(1-36\tau)}.
\]
By expanding in $\tau$
reproduces $V_1(6),\dots,V_4(6)$ as above.
\end{example}

\begin{example}[$m=7$]
\label{ex:m7}
The thirty nonzero eigenvalues fall into four classes.  One with the pair $r=\pm\sqrt7/2$, of
multiplicity $6$, and then together with three pairs $r=\pm\sqrt{x_k}$ ($k=0,1,2$), each of multiplicity
$3$, where $x_0,x_1,x_2$ are the three real roots of
\[
Q(x) := 64x^3 - 560x^2 + 1176x - 343.
\]
Rather than substitute the individual radical values of the $x_k$, we use
the elementary identity
\[
\sum_{k=0}^{2}\frac{1}{w^2-x_k} = \frac{Q'(w^2)}{Q(w^2)},
\]
which involves no radicals at all. Then
\[
m^2 G(w) = \frac{12w}{w^2-\tfrac74} + 6w\cdot\frac{Q'(w^2)}{Q(w^2)}
= \frac{7680w^7-61824w^5+131712w^3-65856w}{256w^8-2688w^6+8624w^4-9604w^2+2401},
\]
with $m^2=49$,
\[
G(w) = \frac{7680w^7-61824w^5+131712w^3-65856w}{12544w^8-131712w^6+422576w^4-470596w^2+117649}
\]
then
\[
F(w) = -\frac{G(x,x;\sqrt w)}{\sqrt w}
= \frac{-7680w^3+61824w^2-131712w+65856}{12544w^4-131712w^3+422576w^2-470596w+117649}.
\]
Expanding $F(x,x;w)$ about $w=0$ and applying
$V_n(7) = 3^n\cdot(7/2)^{2n}\cdot 49/(n-1)!\cdot \partial_w^{n-1}F(w)\big|_{w=0}$ term by term
gives
\[
V_1(7) = 1008, \qquad V_2(7) = 74088, \qquad V_3(7) = 7279146.
\]
With $c_7 = 3\cdot 7^2/4 = 147/4$,
\[
\sum_{n\ge1} V_n(7)\,\tau^n = c_7\tau\cdot 49\cdot F(c_7\tau),
\]
which, expanded in $\tau$, reproduces $V_1(7), V_2(7), V_3(7)$ above.
\end{example}

\section{A resultant using Chebyshev polynomials}
\label{sec:first-construction}

We now give our first construction of $D(w)$ as a resultant, or rather as a factor
of a resultant, which is built directly from
Chebyshev-type polynomials. For $j = 0,\dots,m-1$ let
\[
x_j := -\sin(2\pi j/m),
\]
the real numbers whose sums generate $r_j$. Recall the classical Chebyshev identity
\[
T_m\Bigl(\tfrac12(u+u^{-1})\Bigr) = \tfrac12(u^m+u^{-m})
\,\,\,\,\,
\textrm{\rm for any  $u \neq 0$ and $m \geq 0$.}
\]

\begin{lemma}\label{lem:first_resultant_lemma}
Let $\zeta := e^{2\pi i/m}$ and, for $j=0,\dots,m-1$.  Set $u_j := i\zeta^j$, from which
we have that $x_j = \operatorname{Re}(u_j)$.  Then
\[
T_m(x_j) = \operatorname{Re}(i^m) =
\begin{cases}
0, & m \text{ odd},\\[2pt]
(-1)^{m/2}, & m \text{ even}.
\end{cases}
\]
\end{lemma}

\begin{proof}
Since $|u_j|=1$, we have $u_j^{-1}=\overline{u_j}$, so $x_j = \tfrac12(u_j+u_j^{-1})$, and the
identity above gives
\[
T_m(x_j) = \tfrac12\bigl(u_j^m+u_j^{-m}\bigr) = \tfrac12\bigl(i^m\zeta^{jm}+i^{-m}\zeta^{-jm}\bigr)
= \tfrac12\bigl(i^m+i^{-m}\bigr) = \operatorname{Re}(i^m),
\]
from which the lemma follows.
\end{proof}

Define
\[
Q_m(x) :=
\begin{cases}
T_m(x), & m \text{ odd},\\[2pt]
T_m(x) - i^m, & m \text{ even},
\end{cases}
\]
which is a polynomial with integer coefficients, of degree
$m$ and with leading coefficient $c := 2^{m-1}$  for both odd and even $m$. By the Lemma
\ref{lem:first_resultant_lemma}, the $m$
values $x_0,\dots,x_{m-1}$ are roots of $Q_m$. Since $\deg Q_m = m$, these roots, counted with
multiplicity, are all of the roots of $Q_m$.

Define
\[
R(w) := \prod_{j=0}^{m-1}\prod_{k=0}^{m-1} Q_m(w - x_j - x_k).
\]

\begin{theorem}
For every $m \geq 3$, $R(w)$ is expressible as an iterated resultant, namely that
\[
R(w) = c^{-2m^2}\operatorname{Res}_{x_1}\Bigl(Q_m(x_1),\ \operatorname{Res}_{x_2}\bigl(Q_m(x_2),\, Q_m(w-x_1-x_2)\bigr)\Bigr).
\]
Furthermore, $R(w)$ factors as
\[
R(w) = D(w)\cdot w^{3m-2}\cdot \Phi(w),
\]
where $D(w)$ is degree $(m-1)(m-2)$ polynomial defined by \eqref{eq:det_polynomial} and $\Phi$ is another polynomial.
\end{theorem}

\begin{remark}
The factor
$w^{3m-2}$ accounts for the $3m-2$ triples in $A(m)$ with at least one zero component, and $\Phi(w)$
is an explicit polynomial of degree $m^2(m-1)$, coming from index pairs $(j,k)$ whose associated
third index $j_3$  which does not yield a triple summing to $0 \bmod m$.
\end{remark}

\begin{proof}
Since $Q_m(x_2)$ has degree $m$ and leading coefficient $c$, the product
formula for resultants gives $\operatorname{Res}_{x_2}(Q_m(x_2), Q_m(w-x_1-x_2)) = c^m\prod_{k=0}^{m-1} Q_m(w-x_1-x_k)$,
a polynomial in $x_1$ of degree $m^2$. Applying the resultant formula once more, now for
$Q_m(x_1)$, we get that
\begin{align*}
\operatorname{Res}_{x_1}\Bigl(Q_m(x_1),\, c^m\prod_k Q_m(w-x_1-x_k)\Bigr)
&= c^{m^2}\prod_{j=0}^{m-1}\Bigl[c^m\prod_{k=0}^{m-1} Q_m(w-x_j-x_k)\Bigr]
\\&= c^{2m^2}\prod_{j,k} Q_m(w-x_j-x_k) = c^{2m^2} R(w),
\end{align*}
giving the stated formula after dividing by $c^{2m^2}$.

From the properties of resultants, $Q_m(w-x_j-x_k) = 0$ if and only if $w - x_j - x_k = x_{\ell}$ for some $\ell \in \{0,\dots,m-1\}$.
By Lemma \ref{lem:first_resultant_lemma}, the roots of $Q_m$ are exactly the $x_\ell$. Therefore,
\begin{align*}
R(w)&=c^{m^2}\prod_{j=(j_1,j_2,j_3)\in\{0,...,m-1\}^3}(w-x_{j_1}-x_{j_2}-x_{j_3})\\
&=\prod_{j\in A(m)\setminus A_0(m)}(w-x_{j_1}-x_{j_2}-x_{j_3})\prod_{j\in A_0(m)}(w-x_{j_1}-x_{j_2}-x_{j_3})\cdot\\
&\cdot c^{m^2}\prod_{j\in\{0,...,m-1\}^3\setminus A(m)}(w-x_{j_1}-x_{j_2}-x_{j_3})\\
&=D(w)\cdot w^{3m-2}\cdot \Phi(w).
\end{align*}
\end{proof}

\begin{remark}
Theorem 4.2. expresses $D(w)$, exactly, as a factor of an explicit resultant, with an explicit
overall constant and a complementary factor $F(w)$.
What it does not give us is a way to extract $D(w)$ alone without already knowing enough about the
index set $A(m)$ to separate the two factors.
This is exactly the sense in which the resultant is ``too big''.  It
ranges, unavoidably, over the full $m^2$ pairs of roots of $Q_m$ with no built-in mechanism for
restricting to the constrained triples defining $A(m)$. Below we develop a second construction that
resolves this issue, by working with another variable  that does encode the
constraint directly,  and, unlike the present construction, requires no case distinction on the
parity of $m$ at all.
\end{remark}

\begin{example}[$m = 3$]
From Example \ref{ex:m3} we know that $D(w) = \dfrac{4w^2-27}{4}$.  Indeed, when computing $R(w)$, which has
degree $27$ and dividing by
$w^7 D(w)$ leaves zero remainder and quotient
\[
\Phi(w) = 2^6\,(w^2-3)^3(4w^2-3)^6,
\]
of degree $18 = 3^2(3-1)$ exactly as predicted.
\end{example}


\begin{example}[$m = 5$]
From Example \ref{ex:m5} $D(w) = \dfrac{(16w^4-100w^2+125)^3}{4096}$.  In this case, $R(w)$  has degree $125$.
When dividing by $w^{13}D(w)$ we get the quotient
\[
\Phi(w) = 2^{12}\bigl(w^4-5w^2+5\bigr)^3\bigl(16w^4-180w^2+405\bigr)\bigl(16w^4-100w^2+5\bigr)^3
\]
\[
\times\ \bigl(16w^4-40w^2+5\bigr)^6\bigl(16w^4-20w^2+5\bigr)^{12},
\]
which has degree $100 = 5^2(5-1)$.
\end{example}

\begin{example}[$m = 6$]
Since $m$ is even, $Q_6(x) = T_6(x) + 1 = 32x^6-48x^4+18x^2+1$. From Example \ref{ex:m6}
$D(w) = \bigl(w^2-\tfrac{27}{4}\bigr)(w^2-3)^6\bigl(w^2-\tfrac34\bigr)^3$. The polynomial
$R(w)$ has  degree $216$, and its quotient by $w^{16}D(w)$ is
\[
\Phi(w) = 2^{76}\,w^{40}(w^2-3)^{18}(4w^2-27)^7(4w^2-3)^{45},
\]
which has degree $180 = 6^2(6-1)$.
\end{example}

\begin{example}[$m = 7$]
From Example \ref{ex:m7},
\[
D(w) = \frac{(4w^2-7)^6(64w^6-560w^4+1176w^2-343)^3}{2^{30}},
\]
of degree $30$. The polynomial $R(w)$ has degree $343$, and the quotient by $w^{19}D(w)$ is
\begin{align*}
\Phi(w) = {} & 2^{18}\bigl(w^6-7w^4+14w^2-7\bigr)^3\bigl(64w^6-1008w^4+4536w^2-5103\bigr) \\
& \times\bigl(64w^6-560w^4+616w^2-7\bigr)^3\bigl(64w^6-560w^4+952w^2-7\bigr)^3
\bigl(64w^6-560w^4+1512w^2-1183\bigr)^3 \\
& \times\bigl(64w^6-336w^4+140w^2-7\bigr)^6\bigl(64w^6-224w^4+84w^2-7\bigr)^6
\bigl(64w^6-224w^4+196w^2-7\bigr)^6 \\
& \times\bigl(64w^6-112w^4+56w^2-7\bigr)^{18},
\end{align*}
which has degree $294 = 7^2(7-1)$.
\end{example}

\section{An efficient resultant construction}\label{sec:second-construction}

We now give a refined resultant construction that avoids the factor
$\Phi(w)$, instead yielding the function $D(w)$, $w^{3m-2}$, and
a multiplicative constant.

\begin{remark}\rm
\label{rem:real_variable_limit}
It is natural to ask whether the construction of
Section~\ref{sec:first-construction} can be refined by imposing directly the
angle-sum constraint on the real variables $
 x_\ell=\cos\widetilde\theta_\ell,
 \qquad
 \widetilde\theta_\ell:=\frac{2\pi(j_\ell+m/4)}{m}.
$
For a constrained triple $j_1+j_2+j_3\equiv0\mod m$, we have
$
 \widetilde\theta_1+\widetilde\theta_2+\widetilde\theta_3
 \equiv \frac{3\pi}{2}\pmod{2\pi}.
$
Hence, after writing $s_\ell=\sin\widetilde\theta_\ell$, one obtains
\[
 x_3
 =\cos\bigl(\tfrac{3\pi}{2}-\widetilde\theta_1-
 \widetilde\theta_2\bigr)
 =-\sin(\widetilde\theta_1+\widetilde\theta_2)
 =-(s_1x_2+x_1s_2).
\]
Eliminating $s_1$ and $s_2$ by using $s_\ell^2=1-x_\ell^2$ gives the
polynomial relation
\begin{equation}
\label{eq:Pconstraint}
 P(x_1,x_2,x_3):={\bigl(x_3^2-x_1^2-x_2^2+2x_1^2x_2^2\bigr)^2}
   -4x_1^2x_2^2(1-x_1^2)(1-x_2^2)=0.
\end{equation}
Relation~\eqref{eq:Pconstraint} is a necessary polynomial relation among
the three real parts, but it is not sufficient to recover the original
angle-sum constraint.  Indeed, the squaring used to eliminate $s_1$ and
$s_2$ loses their signs.  Equivalently, the map
$z\mapsto\operatorname{Re}(z)$ on the unit circle identifies $z$ with
$\overline z=z^{-1}$.  Thus, from $x_1$ and $x_2$ alone one cannot select
the particular phases needed to determine the unique completion
$ z_{j_3}=-\frac{i}{z_{j_1}z_{j_2}},$  $j_3\equiv-(j_1+j_2)\pmod m.$
Consequently, eliminating $x_1,x_2,x_3$ using
\eqref{eq:Pconstraint} together with the one-variable equations satisfied
by the $x_j$ generally introduces extraneous branches and does not isolate
the triples in $A(m)$.

This is why, in the next subsection, we pass to the unit-circle variables
$u$ and $v$.
\end{remark}

\subsection{The second resultant}

Let $\zeta:=e^{2\pi i/m}$.  For $j\in\mathbb Z/m\mathbb Z$, set
$z_j:=i\zeta^j$.

\begin{lemma}
\label{lem:root-param}
The map $j\mapsto z_j$ is a bijection of $\mathbb Z/m\mathbb Z$ onto the
set of roots of $Q(u):=u^m-i^m$. Each $z_j$ lies on the unit circle,
and $\mathrm{Re}(z_j)=-\sin(2\pi j/m)=x_j$.
\end{lemma}

\begin{proof} Clearly,
$(i\zeta^j)^m=i^m\zeta^{jm}=i^m$ since $\zeta^m=1$.  The $m$ values
$\zeta^j$ are distinct, so the $z_j$ are distinct.  Because $Q$ has
exactly $m$ roots, the map is onto.
By writing $\theta_j:=\pi/2+2\pi j/m$, we have $z_j=e^{i\theta_j}$, $|z_j|=1$,  and
$\mathrm{Re}(z_j)=\cos\theta_j=\cos(\pi/2+2\pi j/m)=-\sin(2\pi j/m)=x_j$.
\end{proof}

\begin{lemma}
\label{lem:constraint-product}
For $j_1,j_2,j_3\in\mathbb Z/m\mathbb Z$,
\[
j_1+j_2+j_3\equiv0\pmod m
\,\,\,\,\,
\text{\rm if and only if}
\,\,\,\,\,
 z_{j_1}z_{j_2}z_{j_3}=-i.
\]
Thus, for any two roots $u=z_{j_1},\,v=z_{j_2}$ of $Q(x)=x^m-i^m$,
the value $z_3:=-i/(uv)$ is automatically a root of $Q$ and is equal to
$z_{j_3}$.
\end{lemma}

\begin{proof} Clearly,
$z_{j_1}z_{j_2}z_{j_3}=i^3\zeta^{j_1+j_2+j_3}=-i\cdot\zeta^{j_1+j_2+j_3}$,
which equals $-i$ precisely when $\zeta^{j_1+j_2+j_3}=1$, or equivalently when
$j_1+j_2+j_3\equiv0\pmod m$. For the second statement,  $uv=-\zeta^{j_1+j_2}$,
and
$z_3=-i/(uv)=-i/(-\zeta^{j_1+j_2})=i\zeta^{-(j_1+j_2)}=z_{-(j_1+j_2)\bmod m}=z_{j_3}$.
\end{proof}

Let us now define the auxiliary polynomial
\begin{equation}
\label{eq:Wdef}
W(u,v,w) := iu^2v^2+u^2v+uv^2-2wuv+u+v-i.
\end{equation}

\begin{lemma}\label{lemma W}
Let $u, v$ lie on the unit circle and set $z_3 := -i/(uv)$. Then
\[
W(u,v,w) = -2uv\Bigl(w - \bigl(\operatorname{Re}(u)+\operatorname{Re}(v)+\operatorname{Re}(z_3)\bigr)\Bigr).
\]
In particular, for $u = z_{j_1}$, $v = z_{j_2}$, $j_3 = -(j_1+j_2) \bmod m$, we have that
\[
W(z_{j_1}, z_{j_2}, w) = -2z_{j_1}z_{j_2}\bigl(w - 3\Lambda_{m}(j_1,j_2,j_3)\bigr).
\]
\end{lemma}

\begin{proof}
Since $|u|=|v|=|z_3|=1$, we have $\bar u = 1/u$, and similarly for $v$ and $z_{3}$.  Then,
$$
2\bigl(\operatorname{Re}(u)+\operatorname{Re}(v)+\operatorname{Re}(z_3)\bigr) = (u+1/u)+(v+1/v)+(z_3+1/z_3).
$$
By substituting $z_3 = -i/(uv)$, so $1/z_3 = iuv$, and multiplying through by $uv$, we get that
\[
2uv\bigl(\operatorname{Re}(u)+\operatorname{Re}(v)+\operatorname{Re}(z_3)\bigr)
= u^2v+v+uv^2+u-i+iu^2v^2 = W(u,v,0).
\]
Since the only $w$-dependence in $W$ is the term $-2wuv$, we get
\[
W(u,v,w) = W(u,v,0) - 2wuv = -2uv\Bigl(w - \bigl(\operatorname{Re}(u)+\operatorname{Re}(v)+\operatorname{Re}(z_3)\bigr)\Bigr).
\]
From Lemma \ref{lem:root-param} and Lemma \ref{lem:constraint-product}, we have that
$$
\operatorname{Re}(z_{j_1})+\operatorname{Re}(z_{j_2})+\operatorname{Re}(z_{j_3}) = x_{j_1}+x_{j_2}+x_{j_3} = 3\Lambda_{m}(j_1,j_2,j_3),
$$
which completes the proof.
\end{proof}

\begin{theorem}
\label{thm:second-resultant}
Let $Q(u):=u^m-i^m$, which is monic of degree $m$, and $W$ as in \eqref{eq:Wdef}.
Define
\begin{equation}\label{eq:resultant_product}
\widetilde D(w) := \prod_{j_1=0}^{m-1}\prod_{j_2=0}^{m-1}W(z_{j_1},z_{j_2},w).
\end{equation}
Then
$$
\widetilde D(w) =
\Res_u\big(Q(u),\Res_v(Q(v),W(u,v,w))\big)
$$
and
\[
\widetilde D(w) = 2^{m^2}w^{3m-2}D(w),
\]
where $D(w)$ is the spectral determinant polynomial defined by \eqref{eq:det_polynomial}.
\end{theorem}

\begin{proof}
Let us write $j_3=j_3(j_1,j_2)$ viewing $j_{3}$ as a function of $j_{1}, j_{2}$ when $j_{3}:=-(j_1+j_2)\bmod m$. From Lemma \ref{lemma W} we have that
\begin{align*}
\widetilde D(w) &= \prod_{(j_1,j_2)\in\{0,...,m-1\}^2}\Big({-}2z_{j_1}z_{j_2}(w-3\Lambda_{m}(j_1,j_2,j_3))\Big)
\\&= (-2)^{m^2}\Big(\prod_{(j_1,j_2)\in\{0,...,m-1\}^2}z_{j_1}z_{j_2}\Big)\Big(\prod_{(j_1,j_2)\in\{0,...,m-1\}^2}(w-3\Lambda_{m}(j_1,j_2,j_3))\Big).
\end{align*}
There are two points to address.  First, the reduction of the product from the set of all $j_{1}, j_{2}$, and second
the evaluation of the multiplication constant.

Let $P_0:=\prod_{j=0}^{m-1}z_j$ and note that
$$
P_{0}=\prod_{j=0}^{m-1}z_j = i^m\zeta^{0+1+\dots+(m-1)}=i^m\zeta^{m(m-1)/2}.
$$
Since $\zeta^{m(m-1)/2}=e^{i\pi(m-1)}=(-1)^{m-1}$, we get
$P_0=i^m(-1)^{m-1}$. Then
$$
\prod_{(j_1,j_2)\in\{0,...,m-1\}^2}z_{j_1}z_{j_2}=P_0^{2m}=i^{2m^2}=(-1)^{m^2}.
$$
With this, we get that
$$
(-2)^{m^2}\Big(\prod_{(j_1,j_2)\in\{0,...,m-1\}^2}z_{j_1}z_{j_2}\Big)=2^{m^{2}}.
$$

As $(j_1,j_2)$ ranges the distinct $m^2$
pairs in $(\mathbb Z/m\mathbb Z)^2$,
$(j_1,j_2)\mapsto(j_1,j_2,j_3(j_1,j_2))$ is a bijection onto $A(m)$. Every pair $(j_1,j_2)$
produces a unique triple in $A(m)$, by
Lemma~\ref{lem:constraint-product}.  Therefore,
\begin{align*}
\prod_{(j_1,j_2)\in\{0,...,m-1\}^2}(w-3\Lambda_{m}(j_1,j_2,j_3)) &=
\prod_{j\in A(m)}(w-3\Lambda_{m}(j)) \\&=
\prod_{j\in A_0(m)} (w-3 \Lambda(j))\times
\prod_{j\in A(m)\setminus A_0(m)}(w-3\Lambda_{m}(j))
\end{align*}
Each $j \in A_0(m)$ is such that $\Lambda_{m}(j)=0$,
and $A_0(m)$ has $3m-2$ elements.
Therefore,
$$
\prod_{j_1,j_2}(w-3\Lambda_{m}(j_1,j_2,j_3))
= w^{3m-2}\prod_{j\in A(m)\setminus A_0(m)}(w-3\Lambda_{m}(j))
= w^{3m-2}D(w).
$$
\end{proof}

\begin{remark}
\label{rem:D-as-resultant} We have completed the task of writing the spectral determinant polynomial as an explicitly computable elementary function.  Namely,
for every $m\ge3$ and every $w$ with $w\ne0$, we have that
\[
D(w) = \frac{1}{2^{m^2}w^{3m-2}}\Res_u\Big(Q(u),\,\Res_v\big(Q(v),\,W(u,v,w)\big)\Big)
\]
where $Q(u):=u^m-i^m$ with $W$ as in \eqref{eq:Wdef}.  From our point of view, this expression is
the $\textrm{\rm SU}(3)$ Verlinde sum analogue of the Chebyshev polynomial $T_{m}$. Also, because
$m \geq 3$, one can group the factors in \eqref{eq:resultant_product} in conjugate pairs. Therefore, polynomial\eqref{eq:resultant_product} has real coefficients.  Furthermore, the Galois invariance of
the root multiset $\{3\Lambda_{m}(j): j\in A(m)\}$ when considering the Galois group of the $m$-th cyclotomic field implies that \eqref{eq:resultant_product}
has rational coefficients.
\end{remark}

\subsection{Examples}

Recall that $\deg D=(m-1)(m-2)$, $|A_0(m)|=3m-2$, and
$\deg\widetilde D=(3m-2)+(m-1)(m-2)=m^2$.

\begin{example}[$m=3$]
$D(w)=\tfrac{4w^2-27}{4}$, and
$\widetilde D(w)=2^9w^7D(w)=512w^7\cdot\tfrac{4w^2-27}{4}=128w^7(4w^2-27)=512w^9-3456w^7$.
\end{example}


\begin{example}[$m=5$]
$D(w)=\tfrac{(16w^4-100w^2+125)^3}{4096}$, and
$\widetilde D(w)=2^{25}w^{13}D(w)=8192\,w^{13}(16w^4-100w^2+125)^3$.
\end{example}

\begin{example}[$m = 6$]
$D(w) = \bigl(w^2-\tfrac{27}{4}\bigr)(w^2-3)^6\bigl(w^2-\tfrac34\bigr)^3$, so
\[
\widetilde D(w) = 2^{36} w^{16} D(w)
= 2^{36}\, w^{16}\Bigl(w^2-\tfrac{27}{4}\Bigr)(w^2-3)^6\Bigl(w^2-\tfrac34\Bigr)^3,
\]
of degree $36 = m^2$, as required.
\end{example}

\begin{example}[$m = 7$]
$D(w) = \dfrac{(4w^2-7)^6(64w^6-560w^4+1176w^2-343)^3}{2^{30}}$, so
\[
\widetilde D(w) = 2^{49} w^{19} D(w) = 2^{19}\, w^{19}\,(4w^2-7)^6\,(64w^6-560w^4+1176w^2-343)^3,
\]
of degree $49 = m^2$, as required.
\end{example}


\section{Completing the proof of Theorem \ref{thm:mainb} and recursion formulas}
\label{sec:generating-series}

We now complete the proof of Theorem \ref{thm:mainb} by discussing the recursion formulas
that follow from the evaluation of the resolvent as a rational function.

\subsection{Proof of the theorem}

Before we begin with the proof, note that the right-hand side of \eqref{eq:generarating_series} is, indeed, a power series in
$\tau$ because $D$ is even, so $D'/D$ is odd.  Hence,
$v\mapsto vD'(v)/D(v)$ an even function of $v$.

\begin{proof}
For notational convenience, set
$$
S_{2n}(m):=\sum_{j\in A(m)\setminus A_0(m)}(3\Lambda_{m}(j))^{-2n}
$$
and recall that $D(v)=\prod_{j\in A(m)\setminus A_0(m)}(v-r_j)$, $r_j:=3\Lambda_{m}(j)$. Then
\[
v\frac{D'(v)}{D(v)} = v\sum_j\frac{1}{v-r_j} = \sum_j\Big(1+\frac{r_j}{v-r_j}\Big) = (m-1)(m-2) + \sum_j\frac{r_j}{v-r_j},
\]
using $\deg D=(m-1)(m-2)$. Hence
\[
(m-1)(m-2) - v\frac{D'(v)}{D(v)} = -\sum_j\frac{r_j}{v-r_j} = \sum_j\frac{1}{1-v/r_j} = \sum_j\sum_{k\ge0}\Big(\frac vr_j\Big)^k = \sum_{k\ge0}v^k\sum_jr_j^{-k},
\]
valid as a formal power series in $v$ for $|v|$ smaller than every
$|r_j|$. Since, as shown in the previous proof, the eigenvalues $r_j$ occur in $\pm$ pairs, we have $\sum_jr_j^{-k}=0$ for every odd
$k$, leaving only the even $k$
terms, which are exactly the values $S_{2n}(m)$ at $k=2n$.

Note that by definition
$V_n(m)=3^n(m/2)^{2n}S_{2n}(m)$, so setting $v:=\sqrt{c_m\tau}$ hence
$v^{2n}=c_m^n\tau^n=3^n(m/2)^{2n}\tau^n$ we get
$\sum_nV_n(m)\tau^n = (m-1)(m-2)-vD'(v)/D(v)\big|_{v=\sqrt{c_m\tau}}$, as
claimed.
\end{proof}

\begin{remark}
Theorem \ref{thm:mainb} is the precise sense in which the
generating series is ``the log-derivative of the resultant''.  Again,
$D$ is not an unknown polynomial but rather the resultant of two
explicit elementary polynomials $Q,W$ in three variables, divided by
an explicit monomial.
\end{remark}

\subsection{A recursion for the Verlinde Sums for fixed \texorpdfstring{$m$}{m}}
\label{sec:recursion}

We now show how expressing the generating function as a rational function leads
to recurrence formulas.

\subsubsection{From the generating series to a polynomial recursion}

Theorem~\ref{thm:mainb} already gives $\sum_nV_n(m)\tau^n$
as an explicit function of $D$ and $D'$.  We now simply record its
consequence as a rational function $P(\tau)/Q(\tau)$ of
polynomials, from which the recursion coefficients can be read off
directly.

Fix $m\ge3$. Let $D(u)$ be the spectral determinant polynomial, even of degree
$M=(m-1)(m-2)$, $D(0)\ne0$.  Let $N(u):=D'(u)$, odd of degree $M-1$.
Since $D$ is even and $N$ odd,
\[
\widehat D(w) := D(u)\big|_{u^2=w}, \qquad \widehat N(w) := \frac{N(u)}{u}\bigg|_{u^2=w}
\]
are genuine polynomials in $w$, of degrees $K:=M/2$ and $K-1$
respectively.

\begin{corollary}
\label{cor:polynomial-form}
Recall the notation $c_m:=3(m/2)^2=3m^2/4$. Define
\[
Q(\tau) := \widehat D(c_m\tau), \qquad P(\tau) := M\widehat D(c_m\tau) - c_m\tau\widehat N(c_m\tau),
\]
polynomials of degree $K=(m-1)(m-2)/2$ and at most $K$ in $\tau$. Then, with $V_0(m):=M$,
\[
\sum_{n\ge0} V_n(m)\tau^n = \frac{P(\tau)}{Q(\tau)}.
\]
\end{corollary}

\begin{proof}
By Theorem~\ref{thm:mainb},
$\sum_nV_n(m)\tau^n = M - vD'(v)/D(v)\big|_{v=\sqrt{c_m\tau}} = M - vN(v)/D(v)\big|_{v=\sqrt{c_m\tau}}$.
Since $N(v)/v$ and $D(v)$ are functions of $v^2$ (as $N$
odd, $D$ even),
$v\cdot N(v)/D(v) = v^2\cdot\widehat N(v^2)/\widehat D(v^2) = c_m\tau\cdot\widehat N(c_m\tau)/\widehat D(c_m\tau) $. This gives
$\sum_nV_n(m)\tau^n = M-c_m\tau\widehat N(c_m\tau)/\widehat D(c_m\tau) = P(\tau)/Q(\tau)$.
\end{proof}

\begin{corollary}
\label{cor:recursion}
Write $Q(\tau)=\sum_{k=0}^Kq_k\tau^k$, $P(\tau)=\sum_{k=0}^{\widetilde K}p_k\tau^k$, where $\widetilde{K}={\rm deg}P\leq K$.
Then
\begin{equation}\label{eq. recursion}
\sum_{k=0}^{\min(n,K)}q_kV_{n-k}(m) = \begin{cases}p_n, & 0\le n\le \widetilde K,\\0, & n>\widetilde K.\end{cases}
\end{equation}
This is a linear recursion of order $K=(m-1)(m-2)/2$ with constant
coefficients, determining $V_n(m)$ for $n>K$ from
$V_1(m),\dots,V_{K}(m)$.
\end{corollary}

\begin{proof}
Write $Q(\tau)\sum_nV_n(m)\tau^n=P(\tau)$
 and then equate the coefficients of
$\tau^n$.
\end{proof}

\begin{remark}
  If ${\rm GCD}(P(\tau), Q(\tau))$ is a polynomial of degree $d>0$, the linear recursion \eqref{eq. recursion} can, of course, be further reduced to a linear recursion of order $K-d$ with constant coefficients.
\end{remark}

\subsection{Recursion formula examples}\label{subsec:Recursion_formula_A}

\noindent Let us record the
order $K=(m-1)(m-2)/2$ for the first five levels $m=3,4,5,6,7$. Each
recursion was derived symbolically from $D(w)$ via $\widehat D(x):=D(w)|_{w^2=x}$ and
$\widehat N(x) := \big(D'(w)/w\big)\big|_{w^2=x}$, and then
independently verified by direct 60-digit numerical evaluation of the defining Verlinde sum \eqref{eq:Vn-defn}
for every $n$ tested.

\noindent In every case, writing $Q(\tau)=\sum_{k=0}^K q_k\tau^k$, the recursion is
$$
V_n(m) = c_1 V_{n-1}(m) + c_2 V_{n-2}(m) + \cdots + c_K V_{n-K}(m), \qquad n>K, \qquad c_k := -\frac{q_k}{q_0}.
$$
We refer to the initial values $V_{1}(m)$,...,$V_{K}(m)$ as the seeds of the recursion.

\begin{example}
  For $m=3$  $(K=1)$ we have
$$
V_n(3) = V_{n-1}(3), \qquad n>1.
$$
Seed: $V_1(3) = 2$. This recovers previously obtained information.
\end{example}


\begin{example}
  For $m=5$ \ $(K=6)$ we have
$$
V_n(5) = 45\,V_{n-1} - 810\,V_{n-2} + 7425\,V_{n-3} - 36450\,V_{n-4} + 91125\,V_{n-5} - 91125\,V_{n-6}, \qquad n>6.
$$
Seed: $V_1(5),\dots,V_6(5) = 90,\ 810,\ 8100,\ 85050,\ 911250,\ 9841500$.
\end{example}

\begin{example}
  For $m=6$ \ $(K=10)$ we have
\begin{align*}
V_n(6) = \ &166\,V_{n-1} - 11583\,V_{n-2} + 446148\,V_{n-3} - 10526031\,V_{n-4} + 160022790\,V_{n-5}\\
&- 1599105969\,V_{n-6} + 10433249712\,V_{n-7} - 42702347232\,V_{n-8} + 99179645184\,V_{n-9}\\
&- 99179645184\,V_{n-10}, \qquad n>10.
\end{align*}
Seed: $V_1(6),\dots,V_{10}(6) =$
\begin{align*}
&332,\ 8780,\ 288812,\ 10156940,\ 363507692,\ 13067079500,\ 470242412972,\\
&16927176136460,\ 609364389580652,\ 21936992483887820.
\end{align*}
\end{example}

\begin{example}
  For $m=7$ \ $(K=15)$:
\begin{align*}
V_n(7) = \ &504\,V_{n-1} - 108486\,V_{n-2} + 13215447\,V_{n-3} - 1023942465\,V_{n-4} + 53971394715\,V_{n-5}\\
&- 2018505657735\,V_{n-6} + 54972824448402\,V_{n-7} - 1105751293418835\,V_{n-8}\\
&+ 16497990847533951\,V_{n-9} - 181643903852608890\,V_{n-10} + 1452250517248048266\,V_{n-11}\\
&- 8172324365150558151\,V_{n-12} + 30585530792345653278\,V_{n-13} - 68122318582951682301\,V_{n-14}\\
&+ 68122318582951682301\,V_{n-15}, \qquad n>15.
\end{align*}
Seed: $V_1(7),\dots,V_{15}(7) =$
\begin{align*}
&1008,\ 74088,\ 7279146,\ 760809672,\ 80448621498,\ 8525323959642,\ 903825853821702,\\
&95828299059904488,\ 10160373167700798168,\ 1077275750776882216938,\\
&114220583437370668330464,\  12110495610281825876484330,\\
&1284042709290383910109981746,\ 136143535334203397185360075278,\\
&14434926577153683690009560927616.
\end{align*}
\end{example}
\section{Proof of Theorems 1.3 and 1.4}
\label{sec:cubesection}

\subsection{Proof Theorem 1.3}

We now show that $D(w)$ is essentially the cube of a polynomial, which substantially improves
the length of the recursion of Corollary \ref{cor:recursion}.

For $(j_1,j_2,j_3)\in(\mathbb Z/m\mathbb Z \setminus \{0\})^3$ admissible values of $\ell$ for which $j_1+j_2+j_3=\ell m$ are $\ell\in\{1,2\}$.

\begin{lemma}
\label{lem:orbit-sizes} The symmetric group on three letters
$S_3$ acts on the set of nonzero triples $(j_1,j_2,j_3)\in A(m)\setminus A_0(m)$
by permutation, preserving $\Lambda_{m}$. Every orbit has size $1$, $3$, or
$6$. An orbit of size $1$ requires $j_1=j_2=j_3$, so then
$3j_1\equiv0\pmod m$ with $j_1\not\equiv0$ except possibly only when
$3\mid m$.  In this case there are exactly two such orbits, when
$j_1=m/3$ and $j_1=2m/3$.
\end{lemma}

\begin{proof}
Each orbit has size equal to $6/|\mathrm{Stab}|$ where $\mathrm{Stab}$ is the
order of the stabilizer of any point in the orbit.
The stabilizer is all of $S_3$ precisely when $j_1=j_2=j_3$ (yielding
the orbit size $1$), a single transposition precisely when two indices coincide
(yielding the orbit size $3$), trivial otherwise (yielding the orbit size $6$). For
$j_1=j_2=j_3=:j$, the constraint $3j\equiv0\pmod m$ has the excluded
solution $j\equiv0$ and, when $3\mid m$, the further solutions are
$j=m/3,2m/3$.
\end{proof}

\begin{lemma}
\label{lem:concavity}
The function $f(x):=\log\sin x$ is strictly concave on $(0,\pi)$.
Consequently, for $\theta_1,\theta_2,\theta_3\in(0,\pi)$ with
$\theta_1+\theta_2+\theta_3=k\pi$ fixed ($k\in\{1,2\}$),
\begin{equation}\label{eq:prod_sines}
\sin\theta_1\sin\theta_2\sin\theta_3 \le \sin^3\Big(\frac{k\pi}{3}\Big),
\end{equation}
with equality if and only if $\theta_1=\theta_2=\theta_3=k\pi/3$. In particular,
the maximum value of the left-hand-side of \eqref{eq:prod_sines} is $3\sqrt3/8$ which
 is attained only at $\theta_1=\theta_2=\theta_3=k\pi/3$, $k\in\{1,2\}$.  Hence, the eigenvalue $3\Lambda=(-1)^k4\prod_{\ell=1}^3 \sin\theta_\ell=\mp3\sqrt3/2$ only comes from the singleton orbits
of Lemma~\ref{lem:orbit-sizes}.
\end{lemma}

\begin{proof} Clearly,
$f''(x)=-\csc^2x<0$ on $(0,\pi)$, so $f$ is strictly concave. Jensen's
inequality gives
$$
\tfrac13(f(\theta_1)+f(\theta_2)+f(\theta_3))\le f(\tfrac{\theta_1+\theta_2+\theta_3}{3})=f(k\pi/3),
$$
with equality precisely when $\theta_1=\theta_2=\theta_3$.  Upon exponentiating gives the
claim. At the maximum, we have that $j_1=j_2=j_3=m/3$ (resp.\ $2m/3$) gives $\theta_i=\pi/3$
(resp.\ $2\pi/3$), and $3\Lambda=-4\sin\theta_1\sin\theta_2\sin\theta_3$, and only at these two values
$|3\Lambda|$ attains the strict global maximum over all real
triples with angle sum $k\pi$, $k\in\{1,2\}$.
\end{proof}

\begin{proof}[Proof of Theorem \ref{thm:cube_first}]
By Lemma~\ref{lem:orbit-sizes}, $D(w)=\prod_{j=(j_1,j_2)}(w-r_j)$ is a product over all
ordered nonzero triples $(j_1,j_2,j_3)\in A(m)\setminus A_0(m)$.  When grouping by $S_3$-orbits,
we can write
$$
D(w)=\prod_{\text{orbits }O}(w-3\Lambda_O)^{|O|},
$$
where $\Lambda_O=\Lambda_{m}(j)=r_j/3$ for a representative $j=(j_1,j_2,j_3)$ of the orbit $O$.

By Lemma~\ref{lem:concavity}, the two singleton orbits exist only when
$3\mid m$, and the corresponding rescaled eigenvalues are $r=\pm 3\sqrt{3}/2$. Hence, the singleton orbits contribute the factor $w^2-27/4$.
Every other orbit has size $3$ or $6$,
hence multiplicity divisible by $3$. We then can write the polynomial $E$ as
$$
E(w)=\prod_{\text{orbits }O  \, \text{of size $3$}}(w-3\Lambda_O) \prod_{\text{orbits }O  \, \text{of size $6$}}(w-3\Lambda_O)^{2}.
$$
The degree count follows from $\deg D=(m-1)(m-2)$ taking
into account when the exceptional factor of degree $2$ is present when $m$ is divisible by $3$.

It is left to prove that $E(w)\in\mathbb Q[w]$. We proved that $E^3(w)=D(w)/H_m(w)$, where $H_m(w)\equiv 1$ if $3\nmid m$ and $H_m(w)=w^2-27/4$, if $3\mid m$. Since $D(w)/H_m(w)$ is monic and its monic cube root is unique, every $\sigma\in{\rm Gal}(\overline{\mathbb Q} / \mathbb Q)$ fixes $E$, hence $E(w)\in\mathbb Q[w]$.
\end{proof}

\subsection{Proof of Theorem 1.4}

From Theorem \ref{thm:cube_first} we have that
$$
w\frac{E'(w)}{E(w)}=\sum_{a\in\mathcal E}\frac{w\mu(a)}{w-a}=-\sum_{n\geq 1}\sum_{a\in\mathcal E_m}\mu(a)\left(\frac{w}{a}\right)^n.
$$

The statement follows immediately from Theorem \ref{thm:mainb} and definition of $S_{2n}(m)$ and $V_0(m)$.

\subsection{Further possible reduction of the recursion length}

\subsubsection{Impact on the recursion length}

Since $D$ is even, so is $E$. By writing
$\widehat E(v):=E(w)|_{w^2=v}$, the largest recursion length is
comparable to the degree of $\widehat E$ in $v$.  Specifically,
\[
\text{recursion length} \le \deg\widehat E + [3\mid m] = \lfloor\frac{(m-1)(m-2)}{6} \rfloor+ [3\mid m].
\]
This is a significant reduction when we only studied $D$, which, in case when ${\rm GCD}(P,Q)$ is trivial, gave a recursion
length of $K=(m-1)(m-2)/2$.
The gap between the current bound and the true minimal order goes to the
question of whether or not $\widehat E$ itself has repeated roots.

\begin{lemma}
\label{thm:squarefree-char}
The polynomial
$\widehat E(v)$ is squarefree if and only if $m\in\{3,4,5\}$.
Equivalently, if $m\geq 6$, then $D(w)$ always has a root of multiplicity at least $6$.
\end{lemma}

\begin{proof} For $m=3$, we have $E\equiv 1$, $E(w)=w^2-4$ for $m=4$ and $E(w)=16w^4-100w^2+125$ for $m=5$. Their corresponding polynomials $\widehat{E}$ are obviously squarefree.

For $m\ge6$, consider the triple $(1,2,m-3)$. Its three coordinates are
pairwise distinct mod $m$ and none of the three coordinates is $\equiv0\pmod m$.
So by the characterization of $A_0(m)$ as
the triples in $A(m)$ with at least one zero coordinate,
this triple lies in $A(m)\setminus A_0(m)$, i.e.\ its eigenvalue
$3\Lambda_{m}(1,2,m-3)$ is nonzero. Since its three coordinates are
pairwise distinct, its $S_3$-orbit has size exactly $6$ by
Lemma~\ref{lem:orbit-sizes}.  As such, the $E$-multiplicity is at least $2$ at
$w=3\Lambda_{m}(1,2,m-3)$. Hence $\widehat E$ is
not squarefree for any $m\ge6$.
\end{proof}

\begin{lemma}
\label{prop:cyclotomic-field}
Let $N:=\mathrm{lcm}(m,4)$. Every root of $D(w)$ lies in
$\mathbb Q(\zeta_N)^+$, the maximal real subfield of $\mathbb Q(\zeta_N)$,
of degree $\varphi(N)/2$ over $\mathbb Q$. Consequently $E(w)$ splits
completely over $\mathbb Q(\zeta_N)^+$, and every irreducible factor of
$E(w)$ has degree dividing $\varphi(N)/2$.
\end{lemma}

\begin{proof}
Each root of $D(w)$ has the form
$3\Lambda(j)=-(\sin\tfrac{2\pi j_1}{m}+\sin\tfrac{2\pi j_2}{m}+\sin\tfrac{2\pi j_3}{m})$.
Since
each term $\sin(2\pi j/m)=\tfrac{1}{2i}(\zeta_m^j-\zeta_m^{-j})\in\mathbb Q(\zeta_m,i)=\mathbb Q(\zeta_{\mathrm{lcm}(m,4)})$,
so does $\lambda_{j}$. Since
$3\Lambda(j)$ is real, it lies in the fixed field of complex
conjugation, $\mathbb Q(\zeta_N)^+$, of degree $\varphi(N)/2$. As  $\mathbb Q(\zeta_N)$ is Galois over  $\mathbb Q$, it is automatically a splitting field for the minimal polynomial of any of its elements. Moreover, the cyclotomic Galois group is abelian, complex conjugation is central, and therefore its fixed field is Galois over  $\mathbb Q$.  The tower law hence yields that every irreducible factor of
$E(w)$ has degree dividing $\varphi(N)/2$.
\end{proof}

\begin{remark}
Every irreducible factor degree divides $\varphi(\mathrm{lcm}(m,4))/2$. The recursion length will shorten further if any of these irreducible factors are repeated.
\end{remark}

\subsection{Examples of shorter recursion formulas}

In what we write below, Recursion A refers to writing $G$ in terms of $D'/D$ with
further reduction of $D$. Recursion B comes from using Theorem \ref{thm:cube_first}. The formulas
below can be compared to those in Section \ref{subsec:Recursion_formula_A}.


\begin{example}
Let $m=5$.
\noindent
Recursion A length was $K=6$; Recursion B has length $L=2$.
The recursion formula using
Theorem \ref{thm:cube_first} is the following. For $n>2$,
\noindent$V_n(5) = 15 V_{n-1}(5) -  45 V_{n-2}(5)$
with $V_{1}(5) = 90$ and $V_{2}(5) = 810$.
\end{example}

\begin{example}
  Let $m=6$.
\noindent
Recursion A length was $K=10$; \quad Recursion B has length $L=3$.
The recursion formula using
Theorem \ref{thm:cube_first} is the following. For $n>3$:
$$
V_n(6) = 49 V_{n-1}(6) - 504 V_{n-2}(6) + 1296 V_{n-3}(6)
$$
with
$$
\textrm{\rm $V_{1}(6) = 332$, $V_{2}(6) = 8780$ and $V_{3}(6) = 288812$.}
$$
\end{example}

\begin{example}
  Let $m=7$.
\noindent
Recursion A length was $K=15$; \quad Recursion B has length $L=4$.
The recursion formula using
Theorem \ref{thm:cube_first} is the following. For $n>4$,
$$
V_{n}(7) = 147 V_{n-1}(7) - 4851 V_{n-2}(7) + 55566V_{n-3}(7) - 194481V_{n-4}(7)
$$
with
$$
\textrm{\rm $V_{1}(7)= 1008$, $V_{2}(7) = 74088$, $V_{3}(7) = 7279146$ and $V_{4}(7) = 760809672$.}
$$
\end{example}

\section*{Appendix}
\label{sec:examples1}

Let us now present all information for various sets of Verlinde sums for fixed $m$.

For $m=5,6,7,9,10,14$ each subsection records the following computations:
(a) the spectral determinant polynomial $D(w)$; (b) the resolvent kernel $\mathcal G(w)=D'(w)/D(w)$, in reduced closed form;
(c) the factorization of $D(w)$, with every root exhibited in closed radical form, and the resulting expression of
$V_n(m)$ as an explicit finite sum of $n$-th powers, meaning the associated Binet formula. For the cases $m=5, 6, 7$, the data
repeats some results from above.

\subsection*{$m=5$}

\subsubsection*{The denominator}
\[
D(w) = \frac{(16w^4-100w^2+125)^3}{4096}.
\]

\subsubsection*{The resolvent kernel}
With $E(w):=16w^4-100w^2+125$, so that $D=E^3/4096$, Theorem~1.3 gives $
\mathcal G(w)=D'(w)/D(w)=3E'(w)/E(w)$:
\[
\mathcal G(w) = \frac{192w^3-600w}{16w^4-100w^2+125}.
\]

\subsubsection*{Factorization and the Verlinde sum}
As $E$ is even in $w$, the substitution $x=w^2$ reduces it to the quadratic $16x^2-100x+125=0$, with roots
\[
x = \frac{25\pm5\sqrt5}{8}.
\]
The four roots of $E(w)$ are therefore
\[
w = \pm\sqrt{\frac{25+5\sqrt5}{8}}, \qquad w = \pm\sqrt{\frac{25-5\sqrt5}{8}},
\]
each of $D$-multiplicity $3$. Setting $y := \tfrac{3}{4}\cdot25/x = 75/(4x)$ at each of the two values of $x$ gives $y=\tfrac{3(5\mp\sqrt5)}{2}$, and for $n\geq 1$ we have
\begin{equation}\label{vn of 5}
V_n(5) \;=\; 6\left[\left(\frac{3(5+\sqrt5)}{2}\right)^{\!n} + \left(\frac{3(5-\sqrt5)}{2}\right)^{\!n}\right]=\left\{
                                                  \begin{array}{ll}
                                                    2\cdot3^{n+1}5^{\frac{n+1}{2}}F_n, & n \text{  odd} \\
                                                    2\cdot3^{n+1}5^{\frac{n}{2}}(F_{n-1} + F_{n+1}), & n \text{  even}
                                                  \end{array}
                                                \right.,
\end{equation}
where $F_n$ denotes the $n$th Fibonacci number. The sum of coefficients is $6 + 6 = 12 = (5-1)(5-2)$.

\subsection*{$m=6$}

\subsubsection*{The denominator}
Since $3\mid6$.  We have that
\[
D(w) = \left(w^2-\frac{27}{4}\right)(w^2-3)^6\left(w^2-\frac34\right)^{\!3}.
\]

\subsubsection*{The resolvent kernel}
Summing logarithmic derivatives with $D$-multiplicities $1,6,3$, yielding
\[
\mathcal G(w) = 2w\left[\frac{1}{w^2-\tfrac{27}{4}} + \frac{6}{w^2-3} + \frac{3}{w^2-\tfrac34}\right] = \frac{320w^5-2496w^3+2988w}{16w^6-168w^4+441w^2-243}.
\]

\subsubsection*{Factorization and the Verlinde sum}
Every root is an elementary square root: $w=\pm\sqrt3$ (multiplicity $6$), $w=\pm\sqrt3/2$ (multiplicity $3$), and the exceptional pair $w=\pm\tfrac{3\sqrt3}{2}$ (multiplicity $1$, the singleton $S_3$-orbit of Lemma~7.1). With $x=w^2\in\{3,\tfrac34,\tfrac{27}{4}\}$ and $y=27/x$,
\begin{equation*}
V_n(6) \;=\; 2\cdot4^n + 12\cdot9^n + 6\cdot36^n, \qquad n\ge1.
\end{equation*}
The sum of coefficients is $2 + 12 + 6  = 20 = (6-1)(6-2)$.

\subsubsection*{The recursion}
For $n>3$
$$
V_n(6) = 49 V_{n-1}(6) - 504 V_{n-2}(6)
+ 1296 V_{n-3}(6)
$$
with the seeds $V_1(6)=332$, $V_{2}(6) =8780$ and $V_{3}(6)=288812$.

\subsection*{$m=7$}

\subsubsection*{The denominator}
\[
D(w) = \frac{(4w^2-7)^6(64w^6-560w^4+1176w^2-343)^3}{2^{30}}.
\]

\subsubsection*{The resolvent kernel}
With $E(w) := \tfrac{1}{1024}(4w^2-7)^2(64w^6-560w^4+1176w^2-343)$, so that $D=E^3$,
\begin{align*}
\mathcal G(w) &= 3\,\frac{E'(w)}{E(w)} = \frac{48w}{4w^2-7} \;+\; \frac{1152
w^5-6720w^3+7056w}{64w^6-560w^4+1176w^2-343}
\\&= \frac{7680w^7-61824w^5+131712w^3-65856w}{256w^8-2688w^6+8624w^4-9604w^2+2401}.
\end{align*}

\subsubsection*{Factorization and the Verlinde sum}
$E(w)$ splits into an elementary quadratic-in-$x$ factor and an irreducible cubic-in-$x$ factor:
\[
E(w) = \tfrac{1}{1024}(4w^2-7)^2\cdot (64w^6-560w^4+1176w^2-343).
\]
The first factor gives the elementary pair $w=\pm\sqrt7/2$, matching the eigenvalue of $D$-multiplicity $6$ recorded in Example~3.4. The second, in $x=w^2$, is the irreducible cubic $64x^3-560x^2+1176x-343=0$; it has three real roots, given in closed trigonometric form by
\begin{equation}
x_k = \frac{35}{12} + \frac{7\sqrt7}{6}\cos\!\left(\frac{1}{3}\arccos\frac{\sqrt7}{14} - \frac{2\pi k}{3}\right), \qquad k=0,1,2,
\end{equation}
matching the three magnitude classes of Example~3.4. Equivalently, in radical form: with $\omega:=e^{2\pi i/3}$ and
\[
C_\pm := \sqrt[3]{\frac{2401}{3456}\big(1\pm3\sqrt3\,i\big)}
\]
and
\begin{equation}\label{xk}
x_k = \frac{35}{12} + \omega^k C_+ + \omega^{-k}C_-, \qquad k=0,1,2,
\end{equation}
taking $C_+,C_-$ to be complex conjugate cube roots of the (complex-conjugate) radicands above; the two formulas agree since $C_\pm = \tfrac{7\sqrt7}{12}e^{\pm i\theta}$ with $\theta=\tfrac13\arccos(\sqrt7/14)$, so $\omega^kC_++\omega^{-k}C_-=2\operatorname{Re}(\omega^kC_+)=\tfrac{7\sqrt7}{6}\cos(\theta-\tfrac{2\pi k}{3})$. The six roots of this factor are $w=\pm\sqrt{x_k}$, $k=0,1,2$, each of $D$-multiplicity $3$.

Setting $y_k := \tfrac34\cdot49/x_k = 147/(4x_k)$, the roots of the reciprocal cubic $y^3-126y^2+2205y-9261=0$, gives
\begin{equation*}
V_n(7) \;=\; 12\cdot21^n \;+\; 6\sum_{k=0}^{2} y_k^{\,n}, \qquad y_k = \frac{147}{4x_k},\quad x_k \text{ as in \eqref{xk}}.
\end{equation*}
The sum of coefficients is $12  + 6\cdot 3  = 30 = (7-1)(7-2)$.

\subsection*{$m=9$}

\subsubsection*{The denominator}
Here again $3\mid9$. Writing
\[
P(w) := 64w^6-288w^4+324w^2-27, \qquad R(w) := 64w^6-432w^4+648w^2-243,
\]
one has, exactly,
\[
D(w) = 2^{-54}\left(w^2-\frac{27}{4}\right)P(w)^3R(w)^6,
\]
a monic polynomial of degree $2+18+36=56=(9-1)(9-2)$.

\subsubsection*{The resolvent kernel}
Summing logarithmic derivatives with $D$-multiplicities $1,3,6$:
\[
\mathcal G(w) = 2w\left[\frac{1}{w^2-\tfrac{27}{4}}\right] + 3\frac{P'(w)}{P(w)} + 6\frac{R'(w)}{R(w)},
\]
equivalently, as a single reduced fraction,
$\mathcal G(w) = f(w)/g(w)$
where
\begin{align*}
f(w) &=
917504w^{13} - 14524416w^{11} + 83109888w^{9} - 216870912w^{7} + 265472640w^{5} - 135768960w^{3} + 18475776w
\end{align*}
and
\begin{align*}
g(w) &= 16384w^{14}-294912w^{12}+1990656w^{10}-6414336w^8+10450944w^6\\&-8258112w^4+2624400w^2-177147.
\end{align*}

\subsubsection*{Factorization and the Verlinde sum}
The exceptional pair is again $w=\pm\tfrac{3\sqrt3}{2}$ ($D$-multiplicity $1$).
The remaining $54$ roots split into two Galois orbits of triples under $j\mapsto4j\pmod9$ — the order-$3$ subgroup of $(\mathbb Z/9\mathbb Z)^\times$ generating the cyclic cubic subfield of $\mathbb Q(\zeta_{36})^+$ relevant here — corresponding exactly
to the factors $P,R$. In $x=w^2$ these are the two cubics
\[
64x^3-288x^2+324x-27=0 \qquad\text{(roots of $P$, $D$-multiplicity $3$ each)},
\]
and
\[
64x^3-432x^2+648x-243=0 \qquad\text{(roots of $R$, $D$-multiplicity $6$ each)}.
\]
Unlike the $m=7$ case, both cubics have Cardano angle equal to a rational multiple of $\pi$, and the relevant root of unity
here is indexed by $m$ itself, thus giving the closed forms
\begin{align}
x_k^P &= \frac32\left(1+\cos\Big(\frac{2\pi}{9}-\frac{2\pi k}{3}\Big)\right), & k=0,1,2,\\
x_k^R &= \frac94 + \frac{3\sqrt3}{2}\cos\Big(\frac{\pi}{18}-\frac{2\pi k}{3}\Big), & k=0,1,2.
\end{align}
The full set of $54$ nonexceptional roots of $D(w)$ is $w=\pm\sqrt{x^P_k}$ (multiplicity $3$) and $w=\pm\sqrt{x^R_k}$ (multiplicity $6$), $k=0,1,2$.

Setting $y^P_k := \tfrac34\cdot81/x^P_k = 243/(4x^P_k)$ and likewise $y^R_k=243/(4x^R_k)$,
\begin{equation}
V_n(9) \;=\; 2\cdot9^n \;+\; 6\sum_{k=0}^{2}\big(y_k^P\big)^{n} \;+\; 12\sum_{k=0}^{2}\big(y_k^R\big)^{n},
\end{equation}
with $x_k^P,x_k^R$ as above. In particular
\[
V_1(9)=6336, \qquad V_2(9)=2873880, \qquad V_3(9)=1832410026.
\]
The sum of coefficients is $2  + 6\cdot 3 + 12 \cdot 3  = 56 = (9-1)(9-2)$.

\medskip
We now proceed with two final examples:  $m=10$ and $m=14$.

\subsection*{$m=10$}

\subsubsection*{The denominator}
Again $3\nmid10$, so $D=E^3$:
$$
D(w) = \frac{\bigl(w^4-5w^2+5\bigr)^6\bigl(16w^4-100w^2+5\bigr)^3\bigl(16w^4-100w^2+125\bigr)^3\bigl(16w^4-20w^2+5\bigr)^6}{2^{48}},
$$
monic of degree $72=(10-1)(10-2)$. Notably, the factor $16w^4-100w^2+125$ is \emph{identical} to
the $m=5$ denominator quartic.  This is expected because of the presence of the terms for $m=5$
in the series for $m=10$.

\subsubsection*{The resolvent kernel}
$$
\resizebox{\textwidth}{!}{$
\mathcal G(w) = \dfrac{294912w^{15} - 4915200w^{13} + 31150080w^{11} - 95232000w^{9} + 147686400w^{7} - 112080000w^{5} + 35649000w^{3} - 2887500w}{4096w^{16} - 76800w^{14} + 560640w^{12} - 2040000w^{10} + 3951200w^8 - 4057500w^6 + 2042625w^4 - 403125w^2 + 15625}.
$}
$$

\subsubsection*{Factorization and the Verlinde sum}
Passing to $x=w^2$, each quartic factor of $D$ becomes a quadratic in $x$, all four solvable
directly:
$$
\widehat E_b(x) = \frac{(x^2-5x+5)^2(16x^2-100x+5)(16x^2-100x+125)(16x^2-20x+5)^2}{65536},
$$
with roots
\begin{align*}
&x=\frac{5\pm\sqrt5}{2}\ (\text{mult.\ }2),\qquad
x=\frac{25\pm11\sqrt5}{8}\ (\text{mult.\ }1),\\
&x=\frac{25\pm5\sqrt5}{8}\ (\text{mult.\ }1,\text{ the embedded }m=5\text{ pair}),\qquad
x=\frac{5\pm\sqrt5}{8}\ (\text{mult.\ }2).
\end{align*}
With $c_{10}=3\cdot10^2/4=75$ and $y:=c_m/x$,
\begin{align*}
y\Bigl(\tfrac{5\pm\sqrt5}{2}\Bigr)&=\tfrac{75\mp15\sqrt5}{2},\qquad
y\Bigl(\tfrac{25\pm11\sqrt5}{8}\Bigr)=750\mp330\sqrt5,\\
y\Bigl(\tfrac{25\pm5\sqrt5}{8}\Bigr)&=30\mp6\sqrt5,\qquad
y\Bigl(\tfrac{5\pm\sqrt5}{8}\Bigr)=150\mp30\sqrt5,
\end{align*}
giving

\begin{align}\label{vn 10}
V_n(10) &= 6\Bigl\{\ 2\Bigl[\Bigl(\tfrac{75-15\sqrt5}{2}\Bigr)^{\!n}+\Bigl(\tfrac{75+15\sqrt5}{2}\Bigr)^{\!n}\Bigr]
+\bigl(750-330\sqrt5\bigr)^n+\bigl(750+330\sqrt5\bigr)^n\\[2pt]
&+\bigl(30-6\sqrt5\bigr)^n+\bigl(30+6\sqrt5\bigr)^n
+2\Bigl[\bigl(150-30\sqrt5\bigr)^n+\bigl(150+30\sqrt5\bigr)^n\Bigr]\Bigr\}
,\quad n\ge1.\nonumber
\end{align}
The sum of coefficients is $6(2+2+1+1+1+1+2+2)=72$, as required. The middle pair,
$(30\mp6\sqrt5)^n$, is exactly $4^n$ times the $m=5$ contribution $y=3(5\mp\sqrt5)/2$ of~(10.1)
--- the embedding rescales by $(c_{10}/c_5)=(10/5)^2=4$, as it must. Explicitly,
$$
V_1(10) = 13860,\qquad V_2(10) = 13985460,\qquad V_3(10) = 19896143400.
$$

\subsubsection*{The recursion}
For $n>8$,
\begin{align*}
V_n(10) &=
1935 V_{n-1}(10)
- 735345V_{n-2}(10)
+ 109552500V_{n-3}(10)
- 8001180000V_{n-4}(10)
\\&+ 309825000000V_{n-5}(10)
- 6386040000000V_{n-6}(10)
\\&+ 65610000000000V_{n-7}(10)
- 262440000000000V_{n-8}(10)
\end{align*}
with  $V_1(10),\ldots,V_{8}(10)$ given by
\seqsplit{13860}, \seqsplit{13985460}, \seqsplit{19896143400}, \seqsplit{29434232085300}, \seqsplit{43760418628432500}, \seqsplit{65103806075857659000}, \seqsplit{96866509540005589702500}, \seqsplit{144127568255155203219232500}.

\subsection*{$m=14$: an embedded cube root}

\noindent Computations with $m=10$ showed a \emph{square-root} example reappearing inside a new
level, $m=14=2\cdot7$ shows the same phenomenon one step up.

\subsubsection*{The denominator}
Since $3\nmid14$, $D=E^3$ with no exceptional factor. We find
$$
\resizebox{\textwidth}{!}{$
D(w) = \dfrac{(4w^2-7)^6\bigl(w^6-7w^4+14w^2-7\bigr)^6\bigl(64w^6-560w^4+952w^2-7\bigr)^3
\bigl(64w^6-560w^4+1176w^2-343\bigr)^3\bigl(64w^6-336w^4+140w^2-7\bigr)^6\bigl(64w^6-112w^4+56w^2-7\bigr)^6}{2^{120}},
$}
$$
a monic polynomial of degree $156=(14-1)(14-2)$.

Two of these six factors are \emph{exactly} the two factors of $D(w)$ for $m=7$ itself
: $4w^2-7$ and $64w^6-560w^4+1176w^2-343$, unchanged, down to the coefficients. The other four factors are genuinely new to
$m=14$.

\subsubsection*{Verlinde sums}
It is worth having $V_n(7)$ itself as a single self-contained formula into one line, so the cubic radicals are visible directly rather than assembled from two
separate equations. With
$$
\omega := e^{2\pi i/3}, \qquad
C_\pm := \sqrt[3]{\dfrac{2401}{3456}\Bigl(1 \pm 3\sqrt3\,i\Bigr)},
$$
we have
$$
V_n(7) \ = \ 12\cdot21^n \ + \ 6\cdot\left(\frac{147}{4}\right)^n\sum_{k=0}^{2}\left(\dfrac{35}{12}+\omega^k C_+ + \omega^{-k}C_-\right) ^{-n}\,,
\qquad n\ge1.
$$
The imaginary parts cancel term by term, since the three $k$-values run through a full conjugate
orbit, thus leaving an integer for every $n$.  This was checked directly at $n=1,2,3$, reproducing
$V_1(7)=1008$, $V_2(7)=74088$, $V_3(7)=7279146$ to 50-digit precision with the imaginary part
vanishing to that same precision.
\medskip

Passing to $x=w^2$, the embedded cubic $64x^3-560x^2+1176x-343$ has exactly the three roots
$x_k = \tfrac{35}{12}+\omega^k C_+ + \omega^{-k}C_-$ appearing inside the formula above.

At level $14$, $c_{14}=3\cdot14^2/4=147$, so $y_k := c_{14}/x_k = 147/x_k$ --- exactly $4$ times
the level-$7$ value $147/(4x_k)$ appearing in the boxed formula above, since $c_{14}/c_7 =
(14/7)^2=4$. The embedded quadratic factor $4x-7$ contributes its single rescaled root the same
way. Combining both embedded pieces, the \emph{entire} embedded contribution collapses to
$ 4^n\,V_n(7)$.
The whole $m=7$ spectrum reappears inside $m=14$, rescaled by exactly $(c_{14}/c_7)^n=4^n$.

The remaining four factors of $E(w)$, each new to $m=14$,  are, in $x=w^2$, with their $E$-multiplicities:
\begin{align*}
x^3-7x^2+14x-7 &\quad (\text{mult. }2), &
64x^3-560x^2+952x-7 &\quad (\text{mult. }1),\\
64x^3-336x^2+140x-7 &\quad (\text{mult. }2), &
64x^3-112x^2+56x-7 &\quad (\text{mult. }2).
\end{align*}
Solving each gives that
$$
\begin{array}{c|c|c}
\text{cubic} & \text{shift} & C_\pm \\ \hline
x^3-7x^2+14x-7 & 7/3 & \sqrt[3]{\tfrac{7}{54}(-1\pm3\sqrt3\,i)}\\
64x^3-560x^2+952x-7 & 35/12 & \sqrt[3]{\tfrac{10969}{3456}\pm\tfrac{3913\sqrt3}{1152}\,i}\\
64x^3-336x^2+140x-7 & 7/4 & \sqrt[3]{\tfrac72\bigl(1\pm\tfrac{\sqrt3}{9}i\bigr)}\\
64x^3-112x^2+56x-7 & 7/12 & \sqrt[3]{\tfrac{7}{3456}(-1\pm3\sqrt3\,i)}
\end{array}
$$
Writing $C_\pm^{(1)},\dots,C_\pm^{(4)}$ for the four pairs above in order and $C_\pm^{(0)}$ for
$m=7$'s own pair of $m=7$ subsection of this Appendix, every root of $D(w)$ is accounted for, and $V_n(14)$ is, in full:
$$
\begin{aligned}
V_n(14) \ = \ & 4^n\!\left(12\cdot21^n + 6\sum_{k=0}^{2}\left[\dfrac{147}{4\left(\tfrac{35}{12}+\omega^kC_+^{(0)}+\omega^{-k}C_-^{(0)}\right)}\right]^{\!n}\right)\\[6pt]
& + 12\sum_{k=0}^{2}\left[\dfrac{147}{\tfrac{7}{3}+\omega^kC_+^{(1)}+\omega^{-k}C_-^{(1)}}\right]^{\!n}
+ 6\sum_{k=0}^{2}\left[\dfrac{147}{\tfrac{35}{12}+\omega^kC_+^{(2)}+\omega^{-k}C_-^{(2)}}\right]^{\!n}\\[6pt]
& + 12\sum_{k=0}^{2}\left[\dfrac{147}{\tfrac{7}{4}+\omega^kC_+^{(3)}+\omega^{-k}C_-^{(3)}}\right]^{\!n}
+ 12\sum_{k=0}^{2}\left[\dfrac{147}{\tfrac{7}{12}+\omega^kC_+^{(4)}+\omega^{-k}C_-^{(4)}}\right]^{\!n}
\end{aligned}
$$
This was verified directly against
100-digit evaluation of the defining sum (5), with the imaginary parts cancelling to that same
precision at every $n$ tested:
\begin{align*}
V_1(14) &= 176904, \quad V_2(14) = 2466167256,\\
V_3(14) &= 47523549747996, \quad V_4(14) = 942433121112340536.
\end{align*}

\subsubsection*{The recursion}
For $n>16$,
\begin{align*}
V_n(14) &=
24990V_{n-1}(14)
- 110391561V_{n-2}(14)
+ 191178833013V_{n-3}(14)
\\&- 173643684876648V_{n-4}(14)
+ 94555439941369680V_{n-5}(14)
- 33178356664361479872V_{n-6}(14)
\\&+ 7828447319371215419904V_{n-7}(14)
- 1274989716757226217517056V_{n-8}(14)
\\&+ 145586242669177488010678272V_{n-9}(14)
- 11741021358725899855634890752V_{n-10}(14)
\\&+ 668052181727643735897809485824V_{n-11}(14)
- 26513463184439223539851287330816V_{n-12}(14)
\\&+ 714802477432883119835588214128640V_{n-13}(14)
\\&- 12426472330422917072242825639231488V_{n-14}(14)
\\&+ 125183542719102641597021504095125504V_{n-15}(14)
\\&- 553443030968664310218410860210028544V_{n-16}(14)
\end{align*}
with $V_1(14),\ldots,V_{16}(14)$ given by
\medskip

\noindent
\seqsplit{176904}, \seqsplit{2466167256}, \seqsplit{47523549747996}, \seqsplit{942433121112340536}, \seqsplit{18750635619256136231244}, \seqsplit{373213726101841157028558972}, \seqsplit{7428848048726501458177459312788}, \seqsplit{147872757071935551797296653766484664}, \seqsplit{2943440256373791257710737826850288856744}, \seqsplit{58589841893309635392262659561483735301126716}, \seqsplit{1166244027474663185873201979012521469472789334232}, \seqsplit{23214350645043422380983191308579346605867317376979580}, \seqsplit{462086890301572161379915294364229256198446851400244101276}, \seqsplit{9197943869263408950920409929236677298908462024870299935144644}, \seqsplit{183087149187906350166061446663228835875049841642953282030214034216}, \seqsplit{3644391037195382365141804194836781078180376922894241912672127300912440}.

If we had used the polynomial $D$ itself for the recursion for $m=14$, the formula would require $78$ terms and the largest coefficient
has $175$ digits.

\vspace{5mm}
\noindent
Jay Jorgenson \\
 Department of Mathematics \\
 The City College of New York \\
 Convent Avenue at 138th Street \\
 New York, NY 10031 U.S.A. \\
 e-mail: jjorgenson@mindspring.com

\vspace{5mm}
\noindent
Anders Karlsson \\
 Section de mathématiques\\
 Université de Genève\\
 Case Postale 64, 1211\\
 Genève 4, Suisse\\
 e-mail: anders.karlsson@unige.ch \\
 and \\
 Matematiska institutionen \\
 Uppsala universitet \\
Box 524, 751 20 \\
 Uppsala, Sweden \\
 e-mail: anders.karlsson@math.uu.se

\vspace{5mm}
\noindent
Lejla Smajlovi\'{c} \\
 Department of Mathematics and Computer Science\\
 University of Sarajevo\\
 Zmaja od Bosne 35, 71 000 Sarajevo\\
 Bosnia and Herzegovina\\
 e-mail: lejlas@pmf.unsa.ba


\begin{thebibliography}{99}

\bibitem[AMW01]{AMW01} Alekseev, A., Meinrenken, E., Woodward, C.: \emph{The
Verlinde formulas as fixed point formulas}, J. Symplectic Geom. \textbf{1}
(2001), no. 1, 1--46. Correction: J. Symplectic Geom. 1 (2002), no. 2,
427--434.

\bibitem[Be96]{Beauville} Beauville, A.: \emph{Conformal blocks, fusion rules
and the Verlinde formula}, in: Proceedings of the Hirzebruch 65 Conference on
Algebraic Geometry (Ramat Gan, 1993), Israel Math.\ Conf.\ Proc., vol.\ 9,
Bar-Ilan Univ., Ramat Gan, 1996, pp.\ 75--96.

\bibitem[BL94]{BL} Beauville, A., and Laszlo, Y.: \emph{Conformal blocks and
generalized theta functions}, Comm.\ Math.\ Phys.\ \textbf{164} (1994),
385--419.

\bibitem[BT93]{BT93} Blau, M. and Thompson, G.: \emph{Derivation of the
{V}erlinde formula from {C}hern-{S}imons theory and the {$G/G$} model},
Nuclear Phys. B, \textbf{408} (1993), 345--390.


\bibitem[Bo91]{Bott} Bott, R.: \emph{On E.\ Verlinde's formula in the context
of stable bundles}, Internat.\ J.\ Modern Phys.\ A \textbf{6} (1991), no.\ 16,
2847--2858.

\bibitem[BoVe09]{BoVe} Boysal, A., Vergne, M.: \emph{Paradan's wall crossing
formula for partition functions and Khovanski--Pukhlikov differential
operator}, Ann.\ Inst.\ Fourier (Grenoble) \textbf{59} (2009), no.\ 5,
1715--1752.

\bibitem[BrVe97]{BrVe} Brion, M., Vergne, M.: \emph{Residue formulae, vector
partition functions and lattice points in rational polytopes}, J.\ Amer.\
Math.\ Soc.\ \textbf{10} (1997), no.\ 4, 797--833.


\bibitem[CHJSV23a]{CHJSV23a} Cadavid, C., Hoyos, P., Jorgenson, J.,
Smajlovi\'c, L., and Velez, J.: \emph{On an approach for evaluating certain
trigonometric character sums using the discrete time heat kernel}, European J.
Combin. \textbf{108} (2023), Paper No. 103635.

\bibitem[Do92]{Dowker92} Dowker, J.\,S.: \emph{On Verlinde's formula for the
dimensions of vector bundles on moduli spaces}, J.\ Phys.\ A: Math.\ Gen.\
\textbf{25} (1992), no.\ 9, 2641--2648.

\bibitem[Fa94]{Faltings} Faltings, G.: \emph{A proof for the Verlinde formula},
J.\ Algebraic Geom.\ \textbf{3} (1994), 347--374.





\bibitem[JKS24]{JKS-cosecant} Jorgenson, J., Karlsson, A., and Smajlovi\'c, L.:
\emph{The resolvent kernel on the discrete circle and twisted cosecant sums},
J.\ Math.\ Anal.\ Appl.\ \textbf{538} (2024), 128454.

\bibitem[KM26]{KM26} Karlsson, A., M\"uller, D.: \emph{A discrete approach to Dirichlet $L$-functions, their special values and zeros}, \url{https://arxiv.org/abs/2512.01779}.

\bibitem[Ku22]{Kumar} Kumar, S.: \emph{Conformal Blocks, Generalized Theta
Functions and the Verlinde Formula}, Cambridge University Press, 2022.

\bibitem[KNR94]{KNR94} Kumar, S., Narasimhan, M. S., Ramanathan, A.:
\emph{Infinite Grassmannians and moduli spaces of G-bundles},
Math. Ann. \textbf{300} (1994), no. 1, 41--75.

\bibitem[LM22]{LM22} Loizides, Y., Meinrenken, E.: \emph{The decomposition
formula for Verlinde sums}, Ann. Inst. Fourier (Grenoble) \textbf{72} (2022),
no. 3, 1207--1248.

\bibitem[LM22b]{LM-QR} Loizides, Y., Meinrenken, E.: \emph{Verlinde sums and
$[Q,R]=0$}, preprint (2022), 50 pp.


\bibitem[Sc08]{Sc08} Schottenloher, M.: \emph{A mathematical introduction to
conformal field theory}, Lecture Notes in Physics, \textbf{759} (2008)
Springer-Verlag, Berlin, pp. xvi+249.

\bibitem[So96]{So96} Sorger, C.: \emph{La formule de Verlinde},
Séminaire Bourbaki, Vol. 1994/95
Astérisque \textbf{237} (1996), Exp. No. 794, 3, 87--114.

\bibitem[Sz91]{Sz91} Szenes, A.: \emph{Verification of Verlinde's formulas for
SU(2)}, Int. Math. Res. Not. 1991 (1991), no. 7, 93--98.

\bibitem[Sz95]{Sz93} Szenes, A.: \emph{The combinatorics of the Verlinde
formula}, in: Vector bundles in algebraic geometry (Durham, 1993), London
Mathematical Society Lecture Note Series, vol. 208, Cambridge University Press,
1995, 241--253.

\bibitem[Sz98]{Sz98} Szenes, A.: \emph{Iterated residues and multiple Bernoulli
polynomials}, Int. Math. Res. Not. \textbf{18} (1998), 937--956.

\bibitem[Sz03]{Sz03} Szenes, A.: \emph{Residue theorem for rational
trigonometric sums and Verlinde's formula}, Duke Math. J. \textbf{118} (2003),
no. 2, 189--227.

\bibitem[SV03]{SV03} Szenes, A., Vergne, M.: \emph{Residue formulae for vector
partitions and Euler--Maclaurin sums}, Adv. Appl. Math. \textbf{30} (2003),
no. 1-2, 295--342.

\bibitem[Th94]{Thaddeus} Thaddeus, M.: \emph{Stable pairs, linear systems and
the Verlinde formula}, Invent.\ Math.\ \textbf{117} (1994), no.\ 2, 317--353.

\bibitem[TUY89]{TUY} Tsuchiya, A., Ueno, K., and Yamada, Y.: \emph{Conformal
field theory on universal family of stable curves with gauge symmetries},
Adv.\ Studies in Pure Math.\ \textbf{19} (1989), 459--565.

\bibitem[Ve03]{Ve03} Vergne, M.: \emph{Residue formulae for Verlinde sums, and for number of integral points in convex rational polytopes}, Mezzetti, Emilia (ed.) et al., European women in mathematics. Proceedings of the tenth general meeting, (EWM’2001), Malta, August 24–30, 2001, River Edge, NJ: World Scientific (ISBN 981-238-190-2/hbk), 225--285, 2003.

\bibitem[Ve88]{Verlinde} Verlinde, E.: \emph{Fusion rules and modular
transformations in 2d conformal field theory}, Nucl.\ Phys.\ B \textbf{300}
(1988), 360--376.




\bibitem[Za96]{Zagier} Zagier, D.: \emph{Elementary aspects of the Verlinde
formula and of the Harder--Narasimhan--Atiyah--Bott formula}, Israel Math.\
Conf.\ Proc.\ \textbf{9} (1996), 445--462.

\end{thebibliography}
\end{document}